\documentclass[11pt,reqno,twoside]{amsart}
\usepackage{mathrsfs}
\usepackage{amsfonts,amssymb,amsmath,amsthm}
\usepackage[noadjust]{cite}
\usepackage[colorlinks=true,citecolor=red]{hyperref}
\usepackage{tocvsec2}
\usepackage{titletoc}
\usepackage{geometry}
\usepackage{bm}
\usepackage{indentfirst}
\usepackage{graphicx}
\usepackage{float} 
\usepackage{booktabs}
\usepackage{longtable}
\usepackage{subfigure}
\usepackage{stmaryrd}
\usepackage{enumerate}
\usepackage{ulem}
\usepackage{color,soul}
\usepackage{setspace}
\usepackage{stmaryrd}
\usepackage[pagewise]{lineno} 
\usepackage{appendix}
\usepackage{cases}
\usepackage{todonotes}
\usepackage{enumitem}

\numberwithin{equation}{section}
\newtheorem{theorem}{Theorem}[section]
\newtheorem*{maintheorem}{Main Theorem}
\newtheorem{lemma}[theorem]{Lemma}
\newtheorem{corollary}[theorem]{Corollary}
\newtheorem{proposition}[theorem]{Proposition}
\numberwithin{equation}{section}
\newtheorem{remark}[theorem]{Remark}
\newtheorem{definition}[theorem]{Definition}
\newtheorem{example}[theorem]{Example}

\newcommand{\supp}{\operatorname{supp}}
\newcommand{\re}{\operatorname{Re}}
\newcommand{\im}{\operatorname{Im}}
\newcommand{\dist}{\operatorname{dist}}

\newcommand{\Id}{\operatorname{Id}}

\newcommand{\NR}{\operatorname{NR}}

\usepackage{tikz}
\usetikzlibrary{decorations.pathreplacing,calligraphy,shapes}
\tikzstyle{process} = [rectangle, rounded corners, minimum width=5.2cm, minimum height=0.7cm, text centered, draw=black, align=center]
\tikzstyle{smallprocess} = [rectangle, rounded corners, minimum width=4cm, minimum height=0.7cm, text centered, draw=black, align=center]
\tikzstyle{tinyprocess} = [rectangle, rounded corners, minimum width=3cm, minimum height=0.7cm, text centered, draw=black, align=center]
\tikzstyle{point} = [coordinate, on grid]
\tikzstyle{arrow} = [->,>=stealth, thick]
\tikzstyle{dasharrow} = [dashed,->,>=stealth]

\usepackage[flushmargin]{footmisc}
\hypersetup{linkcolor=blue}

\makeatletter
\@namedef{subjclassname@2020}{\textup{2020} Mathematics Subject
	Classification}
\makeatother

\begin{document}
	
\title[Mixing for NLS with degenerate noise]{\Large E\MakeLowercase{xponential mixing for nonlinear} S\MakeLowercase{chr\"odinger equations}\\
\vspace{1mm}\MakeLowercase{perturbed by bounded degenerate noise}}

\subjclass[2020]{}
	
\author[Y. Chen, S. Xiang, Z. Zhang]{Y\MakeLowercase{uxuan} C\MakeLowercase{hen}, S\MakeLowercase{hengquan} X\MakeLowercase{iang}, Z\MakeLowercase{hifei} Z\MakeLowercase{hang}}
		
\address[Yuxuan Chen]{School of Mathematical Sciences, Peking University, 100871, Beijing, China.}
\email{chen\underline{ }yuxuan@pku.edu.cn}

\address[Shengquan  Xiang]{School of Mathematical Sciences, Peking University, 100871, Beijing, China.}
\email{shengquan.xiang@math.pku.edu.cn}

\address[Zhifei  Zhang]{School of Mathematical Sciences, Peking University, 100871, Beijing, China.}
\email{zfzhang@math.pku.edu.cn}

	
	
\begin{abstract}
    We prove the exponential convergence to a unique invariant measure for locally damped nonlinear Schr\"odinger equations, perturbed by bounded noise acting on only two Fourier modes. To tackle the lack of smoothing effect, we introduce asymptotic compactness of linearized system to enhance the coupling method. Inspired by \cite{KNS-20,LWXZZ-24,CXZZ-25}, we establish a new criterion for exponential mixing. Elements from global stability, nonlinear smoothing, and geometric control are combined when applying this criterion.
\end{abstract}
	
\maketitle
	
% \setcounter{tocdepth}{1}
% \tableofcontents

\section{Introduction}

Over the past three decades, the ergodic theory of randomly forced PDEs has attracted a wide range of interest in mathematical physics, with particular focus on the uniqueness and convergence rates of invariant measures for infinite-dimensional dynamical systems under stochastic forcing. Among such systems, nonlinear Schr\"odinger (NLS) equations are of fundamental importance, arising naturally in quantum mechanics, nonlinear optics, and wave propagation in disordered media. The present paper is devoted to the study of exponential mixing for locally damped NLS equations driven by random noise.

A particularly challenging scenario occurs when the noise is {\it highly degenerate}, acting only on specific degrees of freedom. In the context of parabolic PDEs, the problem of unique ergodicity has been extensively studied (see, e.g., \cite{FGRT-15,HM-06,HM-08,HM-11,HMS-11,KNS-20,KNS-20-1,MP-06}). Nevertheless, much less is known for hyperbolic (or dispersive) PDEs. The primary difficulty in the hyperbolic setting stems from the absence of smoothing effect. This typical feature renders most probabilistic criteria developed for parabolic equations ineffective. 

Inspired by the series of works \cite{CLXZ-24,CXZZ-25,CX-26,LWXZZ-24}, our treatment for the lack of smoothing is embodied in the concept of {\it asymptotic compactness}. We believe that the methodology here can be extended to other dispersive models, such as the KdV equation.

\subsection{Main result}
Let us consider the locally damped NLS equation on the torus $\mathbb{T}:=\mathbb{R}/2\pi\mathbb{Z}$, with nonlinear exponent $p\geq 3$ an odd integer, and driven by random force:
\begin{equation}\label{randomNLS}
\left\{\begin{array}{ll}
iu_t+u_{xx}+ia(x)u=|u|^{p-1}u+\eta(t,x),\\
u(0,\cdot)=u_0 \in H^1(\mathbb{T}).
\end{array}\right.
\end{equation}
The assumptions on the damping coefficient $a(x)$ and random noise $\eta(t,x)$ are stated as follows:

\begin{itemize}
\item[$(\mathbf{S1})$] (Localized damping) 
The damping $a: \mathbb{T}\rightarrow \mathbb{R}^+$ is smooth, non-negative, and not identically $0$. In particular, it can be supported in a small open subset of $\mathbb{T}$.

\item[$(\mathbf{S2})$] (Degenerate Haar noise) The random noise $\eta(t,x)$ consists of exactly two spatial Fourier modes $1$ and $e^{ix}$, and takes the form
\[\eta(t,x)=b_0 \eta_0(t)+b_1 \eta_1(t)e^{ix},\]
where $b_0,b_1>0$ are constants, and $\eta_0,\eta_1$ are independent random processes with the same distribution as a fixed process $\tilde{\eta}$ constructed as follows. Let $\{h_0,h_{jl}\}$ denote the $L^\infty$-normalized Haar system defined by \eqref{def_Haar}, then
\[\tilde{\eta}(t)=\sum_{k=0}^\infty (\xi_k^1+i\xi_k^2) h_0(t-k)+\sum_{j=1}^\infty \sum_{l=0}^\infty c_j (\xi_{jl}^1+i\xi_{jl}^2)h_{jl}(t).\]
Here $c_j=cj^{-q}$ with constants $c>0$ and $q>1$ arbitrarily given, and $\xi_k^1,\xi_k^2,\xi_{jl}^1,\xi_{jl}^2$ are real-valued i.i.d.~random variables possessing Lipschitz density function $\rho$ with respect to the Lebesgue measure on $\mathbb{R}$. In addition, we assume that $\supp (\rho)\subset [-1,1]$ and $\rho(0)>0$.

\end{itemize}

The process $\tilde{\eta}$ is referred to as the (complex) Haar noise, which is widely used in engineering and signal processing to model temporally correlated random inputs (see, e.g., \cite{Van-06}). Recently, Haar noise has been incorporated into the ergodic theory of stochastic parabolic PDEs \cite{KNS-20,NZZ-24}, and we are concerned with the hyperbolic counterpart. We also mention that our result remains valid under a slightly more general noise structure (see Section~\ref{sec_applyNLS}).

\medskip

Under the above settings, the NLS system \eqref{randomNLS} is globally well-posed in $H^1(\mathbb{T})$. Moreover, despite the finite depth of temporal correlation, the noise structure guarantees that the solution at integer times $u_n := u(n)\, (n\in \mathbb{N}_0)$ forms a discrete Markov process $(u_n, \mathbb{P}_u)$ in $H^1(\mathbb{T})$. 

Our main result concerning exponential mixing can now be stated as follows.

\begin{maintheorem}
Assume $(\mathbf{S1})$ and $(\mathbf{S2})$ are valid. Then the Markov process $(u_n,\mathbb{P}_u)$ admits a unique invariant measure $\mu\in \mathcal{P}(H^1(\mathbb{T}))$. The support of $\mu$ is compact, and $\mu(H^1\setminus C^\infty)=0$. Moreover, there exist constants $C,\gamma>0$ such that 
\begin{equation}\label{EM-NLS}
\|\mathscr D(u_n)-\mu\|_{L}^*\leq C(1+E(u_0)) e^{-\gamma n}\quad \text{for any }u_0\in H^1(\mathbb{T}),\ n\in \mathbb{N}.
\end{equation}
Here $\|\cdot\|_L^*$ is the dual-Lipschitz distance in $H^1(\mathbb{T})$ (defined by \eqref{dualLip}), $\mathscr D(u_n)$ represents the law of $u_n$, and $E(\cdot)$ stands for the $H^1$-energy (defined by \eqref{energy-function}).
\end{maintheorem}

\begin{remark}
    We emphasize that only two Fourier modes are directly excited by the noise, indicating that even extremely low-dimensional stochastic forcing suffices to drive the solution to the unique equilibrium distribution. To the best of our knowledge, this is the first exponential mixing result for hyperbolic and dispersive PDEs perturbed by highly degenerate noise.\footnote{The reader may compare this result to {\rm\cite{CXZZ-25}}, which establishes exponential mixing for NLS equations driven by random noise localized in space, and acting on all determining modes (i.e.~sufficiently many modes).} We also stress that the noise structure $(\mathbf{S2})$ is independent of the damping coefficient $a(\cdot)$; both the magnitude and spatial support of $a(\cdot)$ can be arbitrarily small. This robustness suggests a potential perspective for the Gibbs measure in the asymptotic regime $a,\eta\to 0$.
\end{remark}

\begin{remark}
    It is noteworthy that the invariant measure $\mu$ is supported on the space of smooth functions $C^\infty(\mathbb{T})$, despite the fact that the solution $u(t)$ does not gain regularity along evolution. This phenomenon is referred to as asymptotic smoothing for deterministic NLS (see, e.g., {\rm\cite{Goubet-00}}). In our random model, it is a consequence of asymptotic compactness (see Proposition~{\rm\ref{prop_H1}}).
\end{remark}

The conclusion remains valid for random initial distribution (independent of the noise) with finite mean of $H^1$-energy, thanks to the Kolmogorov--Chapman relation. The statement and proof are standard and hence omitted (cf.~Theorem~\ref{thm_criterion}).

\subsection{Strategy and ingredients}

The proof of the Main Theorem relies on a new general exponential mixing criterion for Markov processes in non-compact phase spaces (Theorem~\ref{thm_criterion}), which is inspired by prior studies on parabolic systems \cite{KNS-20} and recent works on dispersive models \cite{LWXZZ-24,CXZZ-25}. The key novelty lies in addressing of the lack of smoothing effect, which we achieve through the concept of asymptotic compactness. To apply this criterion to random NLS model, we require several key properties: global stability, nonlinear smoothing, and controllability.

\subsubsection{A general criterion enhanced by asymptotic compactness}

The core of our probabilistic criterion is a coupling method, which, broadly speaking, requires that any two trajectories with sufficiently close initial states to become closer in a probabilistic sense. 

A wide variety of literature \cite{Shi-15,Shi-21,KNS-20,LWXZZ-24,CXZZ-25} has realized that, the coupling condition can be typically derived from a control property. To be more precise, let us formulate the Markov process $(u_n,\mathbb{P}_n)$ induced from the NLS model by
\[u_n=S(u_{n-1},\eta_n),\]
where $\eta_n$ stands for the restriction of the noise $\eta$ to time interval $[n-1,n)$, and $S$ is the time-$1$ solution map of the NLS equation. The desired control property can be formulated as follows: 

\begin{quote}
    Given $u_0\in H^1(\mathbb{T})$ and another initial condition $\tilde{u}_0$ sufficiently close to $u_0$, can we construct, for almost every realization of the noise $\zeta$, a small perturbation $\tilde{\zeta}$ as the control, so that
    \[\|S(u_0,\zeta)-S(\tilde{u}_0,\tilde{\zeta})\|_{H^1}\le q\|u_0-\tilde{u}_0\|_{H^1}\quad \text{for some }q\in (0,1)?\]
\end{quote}

Our previous work \cite{CXZZ-25} reveals that this control property holds provided the reference trajectory issuing from $u_0$ enjoys higher regularity, say $u_0\in H^{1+\sigma}(\mathbb{T})$ for some $\sigma>0$; and it may fail for a specific trajectory in $H^1(\mathbb{T})$ without extra regularity. This necessitates a reduction to a subset of higher regularity. For parabolic PDEs, this follows naturally from smoothing effects. However, for dispersive equations such as the NLS, no smoothing effect is available.

To this end, we adopt the \textbf{exponential asymptotic compactness} (EAC) from \cite{LWXZZ-24,CXZZ-25}. Specifically, for any $\sigma>0$, we can construct a bounded subset $Y \subset H^{1+\sigma} (\mathbb{T})$ such that
\[\dist_{H^1(\mathbb{T})}(u_n,Y)\le C(\|u_0\|_{H^1})e^{-\kappa n}\quad \text{almost surely}.\]
Notably, the trajectory $u_n$ need not enter into $Y$. Still, this exponential attraction permits us to effectively restrict our analysis to the compact set $Y$; cf.~\cite[Proposition 2.4]{LWXZZ-24}.

\medskip

Nevertheless, deducing the coupling condition from control requires a quantitative refinement. In the parabolic case, \cite{HM-06,KNS-20} achieve this by invoking the Moore--Penrose pseudo-inverse, which provides a uniform approximation of the right-inverse of the linearized map $D_{u_0}S(u_0,\zeta)$ on compact subspaces. In the absence of smoothing, this approach breaks down.

To resolve this, we introduce a novel concept, termed {\bf asymptotic compactness of linearization}. Roughly speaking, this property asserts that $D_{u_0}S(u_0,\zeta)$ can be approximated by a compact operator in a suitable asymptotic sense. This idea is inspired by the Nash--Moser iteration; see Section~\ref{subsec_overviewRDS} for more details. 

\subsubsection{Deterministic NLS: global stability, nonlinear smoothing, and geometric control}

To apply our criterion, we need to verify three analytic properties for deterministic NLS; see Figure~\ref{fig_strategy}.

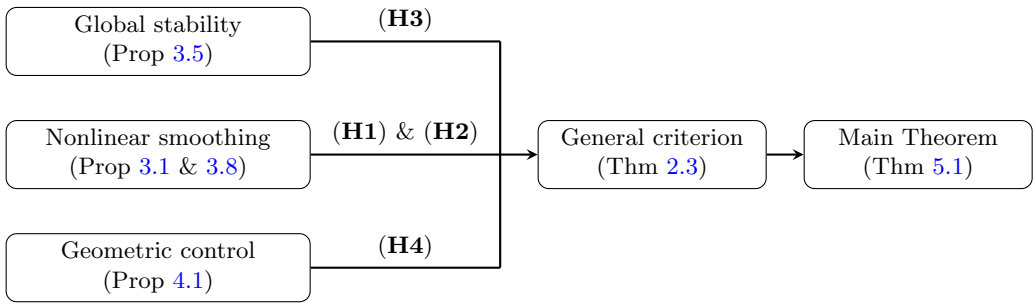
\begin{figure}[t]
    \centering
    \begin{tikzpicture}[node distance=1.5cm]
\footnotesize {
\node (stab) [smallprocess] {Global stability\\(Prop~\ref{prop_stability})};
\node (AC) [smallprocess, below of=stab] {Nonlinear smoothing\\(Prop~\ref{prop_H2} \& \ref{prop_H1})};
\node (ctrl) [smallprocess, below of=AC] {Geometric control\\(Prop~\ref{prop_H4})};

\node (criterion) [tinyprocess, right of=AC, xshift=5cm] {General criterion\\(Thm~\ref{thm_criterion})};
\node (mainthm) [tinyprocess, right of=criterion, xshift=2cm] {Main Theorem\\(Thm~\ref{thm_NLSgeneral})};

\node (pt1) [point, right of=stab, xshift=3cm]{};
\node (pt2) [point, right of=AC, xshift=3cm]{};
\node (pt3) [point, right of=ctrl, xshift=3cm]{};

\draw[thick] (stab) -- node [above, align=center] {$(\mathbf{H3})$} (pt1);
\draw[thick] (AC) -- node [above, align=center] {$(\mathbf{H1})$ \& $(\mathbf{H2})$} (pt2);
\draw[thick] (ctrl) -- node [above, align=center] {$(\mathbf{H4})$} (pt3);
\draw[thick] (pt1) -- (pt2);
\draw[thick] (pt2) -- (pt3);
\draw[arrow] (pt2) -- (criterion);
\draw[arrow] (criterion) -- (mainthm);
% \draw [arrow] (stab) -- node [right, align=left] {$(\mathbf{H2})$ Asymptotic compactness of linearization} (approx-ctrl);
}
\end{tikzpicture}
\caption{The role of three PDE properties in the probabilistic criterion. Here $(\mathbf{H1})$--$(\mathbf{H4})$ refers to various hypotheses for this criterion; see Section~\ref{subsec_RDSsetting}.}
\label{fig_strategy}
\end{figure}

\begin{itemize}
    \item The global stability refers to energy dissipation when there is no external force. Since the damping is merely localized in space, the standard energy method fails. More precisely, the $H^1$-energy functional is defined by
    \begin{equation}\label{energy-function}
        E(v):=\frac{1}{2}\int_{\mathbb{T}}|v|^2+\frac{1}{2}\int_{\mathbb{T}}| v_x|^2+\frac{1}{p+1}\int_{\mathbb{T}}|v|^{p+1},\quad v\in H^1(\mathbb{T}).
    \end{equation}
    Then, provided $\eta\equiv 0$, the energy identity becomes
    \[\frac{d}{dt} E(u(t))=-\int_{\mathbb{T}} a(x) (|u|^2+|u_x|^2+|u|^4)\, dx,\]
    which does not directly guarantee the exponential decay of $E(u(t))$.

    The energy decay under localized damping is a central problem that has been studied by various research groups; see, e.g., \cite{Laurent-ECOCV,LM-23,RZ-09,DGL-06,CXZZ-25}. These results essentially involve controllability, observability and unique continuation from control theory. 
    
    In particular, the corresponding issue in the model of this paper has been tackled in \cite{LM-23} and our prior work \cite{CXZZ-25} by virtue of Carleman estimates; see Proposition~\ref{prop_stability}.

    \item The asymptotic compactness is closely related to nonlinear smoothing, which means the gain of regularity in the nonlinear component (or potential terms) relative to the linear Schr\"odinger evolution. We emphasize that there is no direct regularity gain from the Duhamel convolution $\int_0^t e^{i(t-s)\Delta } f(s)\, ds$ (unlike wave equations \cite{Martirosyan-14,LWXZZ-24}). Basically, the underlying mechanisms for nonlinear smoothing are the dispersive relation and the polynomial structure of nonlinearity, after possibly removing some resonant terms  via suitable phase shift (or gauge transform).
    
    For instance, consider the cubic NLS equation (without damping) on $\mathbb{R}^2$, Bourgain \cite{Bourgain-98} showed that for any initial data $u_0\in H^s(\mathbb{R}^2)$ with $s>3/5$, the nonlinear part of the solution gains extra regularity:
    \[u(t)-e^{it\Delta} u_0\in H^1(\mathbb{R}^2),\]
    although both $u(t)$ and $e^{it\Delta} u_0$ belongs to $H^s(\mathbb{R}^2)$. Similar results are also valid on $\mathbb{R}^d\, (d>2)$ \cite{KV-09}, and on $\mathbb{T}$ up to a phase shift \cite{ET-13,McConnell-22}. More recently, \cite{ST-25} establishes almost sure global-in-time nonlinear smoothing on $\mathbb{T}^2$ for random initial data.
    
    In the recent work \cite{CXZZ-25}, nonlinear smoothing has been utilized to verify EAC; see Proposition~\ref{prop_H1}. And in this paper, we further employ the underlying multilinear estimates in Bourgain spaces to deduce asymptotic compactness of linearization. Namely, consider the linearized equation around the reference trajectory $u$, which reads
    \[iv_t+v_{xx}+ia(x) v=\tfrac{p+1}{2} |u|^{p-1} v+\tfrac{p-1}{2} |u|^{p-3} u^2 \bar{v},\quad v(0,\cdot)=v_0\in H^1(\mathbb{T}),\]
    then the contribution of the potential terms on the right-hand side (up to a phase shift) belongs to a more regular space, provided $u$ possesses extra regularity. We point out that this requirement on $u$ again exposes the necessity of reducing to a compact phase. We refer the reader to Proposition~\ref{prop_H2} and Section~\ref{subsec_ACviaNS} for details.

    \item The approximate controllability, as explained earlier, aims to drive the solution to a given state by finding a suitable external force as a control. Due to our setting on the noise structure, the control is highly degenerate, containing only two Fourier modes. To this end, we exploit the geometric control approach introduced by \cite{AS-06} (see also \cite{Sar-12,BP-25,BCP-25,CoXZ-25} for various NLS equations), which enables us to impose control on the unforced modes indirectly through nonlinear interactions.
    
    Roughly speaking, the nonlinear terms may generate new modes via multiplication, and thus enlarge the space of attainability. This is usually described by successive Lie bracket extensions. To rigorously realize such extension, in the deterministic setting, one can use fast-oscillating control \cite{AS-06,Sar-12}; and in the stochastic setting, one can exploit the roughness of white noise \cite{MP-06}, or Lipschitz-observability for suitable colored noise \cite{KNS-20}.
    
    In Section~\ref{sec_geometriccontrol}, we show that starting with the two Fourier modes in the noise, i.e.~$B_0=\{0,1\}$, it is possible to recursively excite new Fourier modes from
    \[B_n:=B_{n-1}\cup \{2k-l:k\in B_0,\ l\in B_{n-1}\},\]
    which implies the desired controllability since the iterated extensions $B_n$ span all Fourier modes, in the sense that $\mathbb{Z}=\bigcup_{n=0}^\infty B_n$. 
\end{itemize}

\medskip

Finally, we emphasize that the extra regularity in EAC plays a dual role: it is essential both for the probabilistic coupling argument and for the deterministic analysis.

\subsection{Related literature}

We briefly review previous results on the unique ergodicity and mixing of PDEs perturbed by additive noise.

\subsubsection{Parabolic equations} Significant progress has been made in last decades on the ergodic and mixing properties of random parabolic PDEs. 

For the 2D Navier--Stokes system, early works (e.g., \cite{FM-95, EMS-01, BKL-02}) address cases where all determining modes are directly perturbed. Hairer and Mattingly \cite{HM-06, HM-08} later applied Malliavin calculus and hypoelliticity to treat systems with highly degenerate white-in-time noise. More recently, Kuksin, Nersesyan and Shirikyan \cite{KNS-20,KNS-20-1} established a controllability-based method to handle bounded degenerate colored noise. Shirikyan \cite{Shi-15, Shi-21} also proposed a controllability approach for systems with physically localized bounded random forces. Mixing results for Navier--Stokes system on unbounded domains have also been obtained \cite{Ner-22, NZ-24}. 

For other parabolic equations with degenerate noise, see, e.g., \cite{FGRT-15} for Boussinesq equation with white noise and \cite{Ner-24} for complex Ginzburg--Landau equation with localized noise.

\subsubsection{Hyperbolic equations} In contrast, much less is known about the ergodicity for hyperbolic PDEs. And most existing results concern equations with full-domain damping.

For wave equations, one early result on unique ergodicity was given by Barbu and Da Prato \cite{BDP-02}. Later, Martirosyan \cite{Martirosyan-14} established exponential mixing for 3D wave equations with white-in-time noise, with typically sub-cubic nonlinearity. More recently, the last two authors and Liu, Wei, Zhao \cite{LWXZZ-24} proved exponential mixing of 3D cubic wave equations with localized damping and localized colored noise, by introducing EAC and controllability results.

As for Schr\"odinger equations, Debussche and Odasso \cite{DO-05} obtained polynomial mixing on an interval. Recently, the authors and Zhao \cite{CXZZ-25} established the first exponential mixing result for NLS equations. The unique ergodicity in $\mathbb{R}^d$ is known only for large damping \cite{EKZ-17,NS-25}.

In addition, the unique ergodicity for stochastic KdV equation can be found in \cite{GMR-21}, which also deduced exponential mixing when the damping is sufficiently large.

\medskip

Finally, we also refer the reader to \cite{BT-08,BT-24,GOT-22,BDNY-24,GKO-24} and references therein for other related topics on random dispersive equations.

\subsection{Organization} In Section~\ref{sec_RDS}, we establish a general criterion for the exponential mixing of Markov processes, combining two key ingredients: EAC and asymptotic compactness of linearization. Then we turn to the study of deterministic NLS equations, which serves to verify various hypotheses in our criterion. We study asymptotic compactness for linearized system in Section~\ref{sec_globaldynamic}, and carry out the geometric control analysis in Section~\ref{sec_geometriccontrol}. Finally, we prove a generalization of the Main Theorem in Section~\ref{sec_applyNLS}. Some auxiliary results and proofs are collected in the Appendix for the reader's convenience.

\subsection{Notation and convention}\label{subsec_notation}

We gather here some repeatedly used notations in this paper.

\medskip

\noindent $\bullet$ \textit{Fourier analysis.} For $k\in \mathbb{Z}$, the $k$-th Fourier mode on $\mathbb{T}$ is $e_k:=\frac{1}{\sqrt{2\pi}} e^{ikx}$. For any function $u\colon \mathbb{T}\to \mathbb{C}$, the corresponding Fourier coefficients are $\mathcal{F}u(k)=\hat{u}(k)=(u,e_k)_{L^2(\mathbb{T})}$,
where $(\cdot,\cdot)_{L^2(\mathbb{T})}$ is the complex $L^2$-inner product: $(f,g)_{L^2(\mathbb{T})}=\int_{\mathbb{T}} f(x)\overline{g(x)}\, dx$. The Sobolev space $H^s(\mathbb{T})\, (s\in \mathbb{R})$ is equipped with standard norm $\|u\|_{H^s}^2:=\sum_{k\in \mathbb{Z}} \langle k\rangle^{2s}|\hat{u}(k)|^2$, where $\langle x\rangle:=\sqrt{1+|x|^2}$. When there is no danger of ambiguity, we simply write $H^s$ for $H^s(\mathbb{T})$.

We write $S_a(t)$ (and $S(t)$) for the $C_0$-group of operators on $H^1(\mathbb{T})$ generated by $i\partial_x^2-a(x)$ (and $i\partial_x^2$, respectively). In other words, the solution of $iu_t+\Delta u+ia(x)u=0$ is $u(t)=S_a(t)u_0$.

\medskip

\noindent $\bullet$ \textit{Random variables.} Let $X$ be a Polish space (i.e.~separable metric space). The distance from $x\in X$ to $A\subset X$ is $\dist_X(x,A)=\inf\{d(x,a):a\in A\}$. The Borel $\sigma$-algebra is $\mathcal{B}(X)$. The space of bounded continuous functions on $X$ is $C_b(X)$, equipped with the supremum norm $\|f\|_\infty=\sup_X |f|$. And bounded Lipschitz functions constitute $L_b(X)$, with norm $\|f\|_{L_b(X)}=\|f\|_\infty+\sup_{x\not =y} \frac{|f(x)-f(y)|}{d(x,y)}$.

The law of an $X$-valued random variable $\eta$ is denoted by $\mathscr{D}(\eta)$, which belongs to the space of probability measures $\mathcal{P}(X)$. In the context, we use $\eta$ and its variants $\eta',\hat{\eta},\tilde{\eta}$ to denote random elements, and employ the letter $\zeta$ for deterministic ones to avoid possible confusions. A Borel map $f\colon X\to Y$ between Polish spaces pushes $\mu\in \mathcal{P}(X)$ forward to $f_* \mu\in \mathcal{P}(Y)$ via $f_*\mu(\cdot)=\mu (f^{-1}(\cdot))$. In particular, the law of $f(\eta)$ coincides with $f_* \mathscr{D}(\eta)$.

The weak convergence in $\mathcal{P}(X)$ can be metrized by the dual-Lipschitz distance:
\begin{equation}\label{dualLip}
    \|\mu-\nu\|_L^*:=\sup_{\|f\|_{L_b(X)}\le 1} |\langle f,\mu\rangle-\langle f,\nu\rangle|,\quad \mu,\nu\in \mathcal{P}(X).
\end{equation}
A coupling between $\mu$ and $\nu$ is a pair of $X$-valued random variables with marginal distributions equal to $\mu$ and $\nu$, respectively. The set of all couplings between $\mu$ and $\nu$ is denoted by $\mathscr{C}(\mu,\nu)$.

\medskip

\noindent $\bullet$ \textit{Functional analysis.} Throughout this paper, we regard any complex Hilbert space $X$  as a real Hilbert space, by replacing the complex inner product $(\cdot,\cdot)$ with $\re\, (\cdot,\cdot)$.\footnote{With complex-valued NLS equations in mind, the scalar field for Hilbert spaces in consideration should be $\mathbb{C}$ rather than $\mathbb{R}$. However, the solution map is not complex-differentiable, due to the presence of complex conjugation in nonlinear term $|u|^{p-1} u$. For this reason, we need to allow real-linear maps in the sequel.} Denote by $B_X(R)$ the open ball of radius $R$ centered at the origin. We write $\mathcal{L}(X,Y)$ for  bounded (real-)linear operators from $X$ to another Hilbert space $Y$, and simply write $\mathcal{L}(X)$ when $Y=X$.

\medskip

\noindent $\bullet$ \textit{Constants.} Various constants $C$ may change from line to line. The dependence on parameters are indicated by $C(\cdot)$ or displayed in the subscript.

\section{Mixing of random dynamical systems}\label{sec_RDS}

In this section we establish an abstract EAC-based criterion for exponential mixing, which is inspired by \cite{KNS-20,LWXZZ-24,CXZZ-25} and forms the probabilistic backbone of this paper. In order to compensate for the lack of smoothing effect, we propose asymptotic compactness of linearization as a new element, which allows us to enhance the classical coupling approach. We will apply this criterion in Section~\ref{sec_applyNLS} to conclude the Main Theorem.

The schematic road map for this criterion can be depicted as in Figure~\ref{fig_outline}. A more detailed overview is presented in Section~\ref{subsec_overviewRDS}

\begin{figure}[t]
    \centering
    \begin{tikzpicture}[node distance=1.5cm]
\footnotesize {
\node (p-inv) [process] {Approximate inverse (Lem~\ref{lem_inverse})};
\node (approx-ctrl) [process, below of=p-inv] {Control property (Lem~\ref{lem_approxcontrol})};
\node (coupling) [process, below of=approx-ctrl] {Coupling method (Lem~\ref{lem_coupling})};
\node (EMonY) [process, below of=coupling] {Exp.~mixing on $Y$ (Prop~\ref{prop_restricY})};
\node (EMonX) [process, below of=EMonY] {Exp.~mixing on $X$ (Thm~\ref{thm_criterion})};

\node (pt1) [point, left of=approx-ctrl, xshift=-1.32cm]{};
\node (pt2) [point, left of=EMonY, xshift=-1.3cm]{};

\draw[decorate,decoration={calligraphic brace,amplitude=4mm,mirror}, thick] (pt1) -- (pt2) node[black,midway,xshift=-2.2cm, align=center]{
Hypotheses $(\mathbf{H3})$--$(\mathbf{H5})$};

% \node (pt3) [point, left of=EMonY, xshift=-1.32cm, yshift=-0.05cm]{};
% \node (pt4) [point, left of=EMonX, xshift=-1.3cm]{};

% \draw[decorate,decoration={calligraphic brace,amplitude=4mm,mirror}, thick] (pt3) -- (pt4) node[black,midway,xshift=-2.5cm, align=center]{Hypothesis $(\mathbf{H1})$ EAC\\Same as \cite[Proposition 2.4]{LWXZZ-24}};

\draw [arrow](p-inv) -- node [right, align=left] {$(\mathbf{H2})$ Asymptotic compactness of linearization} (approx-ctrl);
\draw [arrow](approx-ctrl) -- (coupling);
\draw [arrow](coupling) -- (EMonY);
\draw [arrow] (EMonY) -- node [right, align=left] {$(\mathbf{H1})$ Exponential asymptotic compactness} (EMonX);

% \node (new) [idea, right of=coupling, xshift=3cm] {New ingredients};
% \node (pt3) [point, right of=approx-ctrl, xshift=3cm, yshift=1cm]{};
% \node (pt4) [point, right of=EMonY, xshift=3cm, yshift=-0.5cm]{};
% \draw [arrow](new) -- (pt3);
% \draw [arrow](new) -- (pt4);
}
\end{tikzpicture}
\caption{Outline for the general criterion. We point out that the first step involves the new concept ``asymptotic compactness of linearization". Meanwhile, the idea of two implications in the middle are similar to parabolic PDEs (see, e.g., \cite{Shi-15,Shi-21,KNS-20,KNS-20-1}), and the last step relying on EAC follows from \cite{LWXZZ-24}.}
\label{fig_outline}
\end{figure}
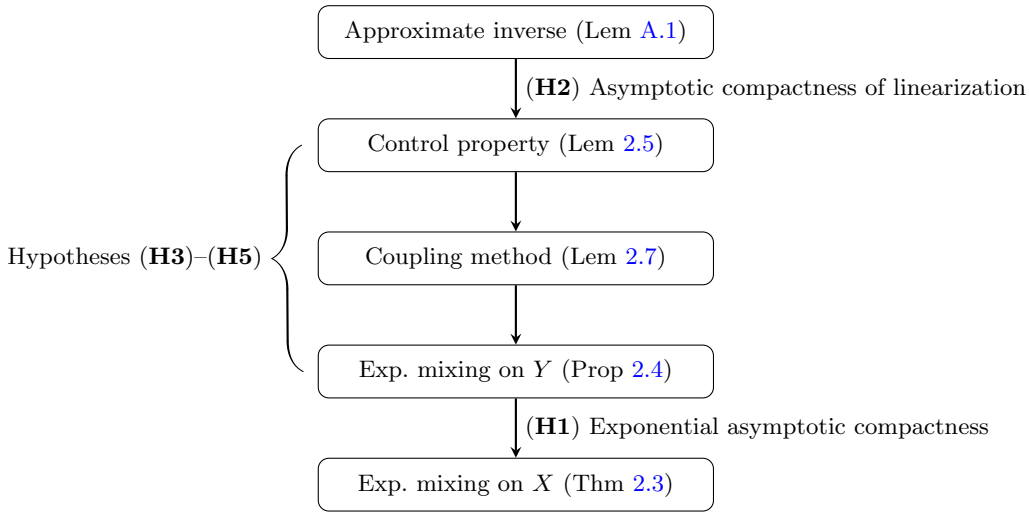

\subsection{New general criterion}\label{subsec_RDSsetting}

% Randomly force PDEs naturally give rise to random dynamical systems, which, in turn, lead to Markov processes; see, e.g., \cite{KS-12}. 
Our setup for random dynamical systems consists of the following data, which is a modification of those in \cite{KNS-20,LWXZZ-24,CXZZ-25}:

\begin{itemize}
    \item Separable real Hilbert spaces $(X,\|\cdot\|)$, $(E,\|\cdot\|_E)$ and $(V,\|\cdot\|_V)$. Here $X$ stands for the phase space, $E$ represents the noise space, and $V$ is compactly embedded into $X$. In addition, we assume $Y$ is a compact subset of $X$.\footnote{We inform the reader to distinguish compact subset $Y$ and the Hilbert space $V$ compacted embedded in $X$. A typical situation is $Y$ being a bounded subset of $V$. But this is not necessarily the case. In our application to NLS model (see Section~\ref{sec_applyNLS}), $X=H^1(\mathbb{T})$, and $Y$ is bounded in $H^3(\mathbb{T})$, while $V=H^{5/4}(\mathbb{T})$.}  Both $V$ and $Y$ reflect asymptotic compactness of the system from different aspects; see hypotheses $(\mathbf{H1})$ and $(\mathbf{H2})$.
    
    \item The evolution is characterized by a smooth (in Fr\'echet sense) map $S\colon X\times E\to X$. We assume that $S$ has bounded second-order derivatives on every bounded subset of $X \times E$, and specifically $S$ is locally Lipschitz. We denote the partial derivatives with
    \[D_x S(x,\zeta)\colon X\to X\quad \text{and}\quad D_{\zeta} S(x,\zeta)\colon E\to X.\]
    
    % For the sake of clarity, we remind the reader that, although $X$ and $E$ may be complex Hilbert spaces, we deliberately forget the complex structure and treat them as Hilbert spaces over $\mathbb{R}$ (see the Convention in Section~\ref{subsec_notation}). In particular, $D_x S(x,\zeta)$ and $D_\zeta S(x,\zeta)$ are merely real-linear operators.

    \item A sequence of $E$-valued i.i.d.~random variable $(\eta_n)_{n\in \mathbb{N}}$, with common law $\ell\in \mathcal{P}(E)$. We assume the support of $\ell$, denoted with $K=\supp(\ell)$, is compact in $E$.
\end{itemize}

Since $S$ is continuous and $\eta_n$ are i.i.d., the random dynamical system generates a Feller family of discrete Markov processes $(x_n,\mathbb{P}_x)$ on $X$ (see, e.g., \cite[Section 1.3]{KS-12}), which is formulated by
\begin{equation}\label{Markovchain}
    x_n=S(x_{n-1},\eta_n),\quad x_0=x\in X.
\end{equation}
We also introduce a self-explanatory notation to indicate initial state and random input:
\[x_n=S_{n}(x;\eta_1,\dots,\eta_n).\]

For this Markov chain $(x_n,\mathbb{P}_x)$, we denote the corresponding expected values by $\mathbb{E}_x$, and the Markov transition probabilities by $P_n(x,\cdot )$, i.e.~$P_n(x,B):=\mathbb{P}_x(x_n\in B)$ for $B\in\mathcal{B}(X)$. The standard notation for the Markov semigroup $P_n\colon C_b(X)\rightarrow C_b(X)$ is employed:
    \begin{equation*}
        P_nf(x):= \int_X f(x')\, P_n(x,dx')\quad \text{for }f\in C_b(X).
    \end{equation*}
    And $P^*_n\colon \mathcal{P}(X)\rightarrow \mathcal{P}(X)$ refers to the dual semigroup, which is defined via
    \[P_n^* \mu (B)=\int_X P_n(x,B)\, \mu(dx)\quad \text{for }B\in \mathcal{B}(X).\]
    In particular, for any initial state $x\in X$, we have $\mathscr{D}(x_n)=P_n^* \delta_{x}$. For convenience, we often omit the subscript when $n=1$, and write $P=P_1$ and $P^*=P_1^*$. A probability measure $\mu\in \mathcal{P}(X)$ is called {\it invariant} for this Markov process if $P^*\mu=\mu$. And we say a subset $A\subset X$ is invariant, if
    \[S(x,\zeta)\subset A\quad \text{for any } x\in A,\ \zeta\in K.\]

    \medskip
        
    Below is a list of hypotheses regarding our criterion. Without loss of generality, we assume the parameters $q_0\in (0,1)$ involved in hypotheses $(\mathbf{H1})$--$(\mathbf{H3})$ are the same.

\begin{itemize}
    \item[(\textbf{H1})] \textbf{Exponential asymptotic compactness (EAC).} The compact subset $Y\subset X$ is invariant, and there exists
    a constant $q_0\in (0,1)$, and an increasing function $G\colon [0,\infty)\to [0,\infty)$, such that for any $x\in X$, $n\in \mathbb{N}$, and $\zeta_1,\dots ,\zeta_n\in K$, we have
    \[\dist_X (S_n(x;\zeta_1,\dots, \zeta_n),Y)\le q_0^n G(\|x\|).\]

    \item[(\textbf{H2})]  \textbf{Asymptotic compactness of linearization.} There exist constants $C_0>0$, $q_0\in (0,1)$, and a map $T\colon Y\times K\to \mathcal{L}(X,V)$\footnote{Since $V\hookrightarrow X$, with a slight abuse of notation, $T(y,\zeta)$ is also a member of $\mathcal{L}(X)$ (see, e.g., \eqref{ACL-2} and \eqref{ACL-3}).}, such that for any $y\in Y$ and $\zeta\in K$, we have
    \begin{equation}\label{ACL-1}
        \|T(y,\zeta)\|_{\mathcal{L}(X,V)}\le C_0,
    \end{equation}\begin{equation}\label{ACL-2}
        \|D_y S(y,\zeta)-T(y,\zeta)\|_{\mathcal{L}(X)}\le q_0.
    \end{equation}
    Moreover, with respect to the $\|\cdot \|_{\mathcal{L}(X)}$-norm, $T(y,\cdot)$ is Lipschitz-continuous:
    \begin{equation}\label{ACL-3}
        \|T(y,\zeta)-T(y,\zeta')\|_{\mathcal{L}(X)}\le C_0 \|\zeta-\zeta'\|_E\quad \text{for any }y\in Y,\ \zeta,\zeta'\in K.
    \end{equation}

    \item[(\textbf{H3})] \textbf{Dissipativity.}
    There exists $\tilde{y}\in Y$, $m_0\in \mathbb{N}$, $q_0\in (0,1)$ and $\tilde{\zeta}_1,\dots ,\tilde{\zeta}_{m_0}\in K$, such that
    \begin{equation}\label{Dis-1}
        \|S_{m_0}(y;\tilde{\zeta}_1,\dots ,\tilde{\zeta}_{m_0})-\tilde {y}\|\le q_0\|y-\tilde{y}\|\quad \text{for any }y\in Y.
    \end{equation}

    \item[(\textbf{H4})] \textbf{Approximate controllability along trajectory.} For any $y\in Y$, define
    \begin{equation}\label{K^ydef}
        K^y:=\{\zeta\in K:D_\zeta S(y,\zeta)\colon E\to X\text{ has dense image in }X\}.
    \end{equation}
    Then we have $\ell(K^y)=1$ (recall $\ell=\mathscr{D}(\eta_n)$).

    \item[(\textbf{H5})] \textbf{Decomposability of the noise.} The random noise $\eta_n$ has the structure
    \[\eta_n=\sum_{k\in \mathbb{N}} b_k\xi_{kn} \psi_k.\]
    Here $(\psi_k)_{k\in \mathbb{N}}$ is an orthonormal basis of $E$, the constants $b_k>0$ satisfy $\sum_{k\in \mathbb{N}} |b_k|^2<\infty$, 
    and $\xi_{kn}$ are independent real-valued random variables possessing Lipschitz-continuous density function $\rho_k$ supported in $[-1,1]$ with respect to the Lebesgue measure on $\mathbb{R}$.\footnote{Note that $(\mathbf{H5})$ ensures the support of $\ell\in \mathcal{P}(E)$ to be compact, via the diagonal argument.}
\end{itemize}

% Before presenting our general criterion, we compare our hypotheses to those in \cite{KNS-20}.

\begin{remark}
    The core distinction between our criterion and existing ones for parabolic PDEs lies in the regularity assumption. More precisely, {\rm\cite{KNS-20}} assumes that the system evolution exhibits a gain in regularity, in the sense that the $S$ maps $X\times E$ to $V$. In contrast, to deal with dispersive PDEs, we introduce new asymptotic compactness conditions $(\mathbf{H1})$ and $(\mathbf{H2})$. Specifically, the operator $T\in \mathcal{L}(X,V)$ serves as a compact approximation of the non-compact operator $D_y S(y,\zeta)$.
\end{remark}

\begin{remark}
    We refer the reader to Figure~{\rm\ref{fig_strategy}} for the relation between these hypotheses to the NLS model. The corresponding PDE properties will be established in Sections~{\rm \ref{sec_globaldynamic}}--{\rm \ref{sec_geometriccontrol}}. 
\end{remark}

\medskip

Now we state the general criterion for exponential mixing.

\begin{theorem}\label{thm_criterion}
    Under the above settings, assume the hypotheses $(\mathbf{H1})$--$(\mathbf{H5})$ are valid. Then the Markov process $(x_n,\mathbb{P}_x)$ defined by \eqref{Markovchain} admits a unique invariant measure $\mu\in \mathcal{P}(X)$. Moreover, the support of $\mu$ is contained in $Y$, and there exist constants $C,\gamma>0$ such that
    \begin{equation}\label{EM-1}
        \|P_n^* \lambda-\mu\|_L^*\le Ce^{-\gamma n}\left(1+\int_X G(\|x\|)\, \lambda(dx)\right)\quad \text{for any }\lambda\in \mathcal{P}(X),\ n\in \mathbb{N}.
    \end{equation}
\end{theorem}

Inspired by \cite{LWXZZ-24}, it turns out that owing to the EAC hypothesis $(\mathbf{H1})$, the problem can be reduced to the compact invariant set $Y$. More precisely, to prove Theorem~\ref{thm_criterion}, it suffices to establish the following proposition with $\lambda$ specified as $\delta_{y_0}$.

\begin{proposition}\label{prop_restricY}
    Under the assumptions of Theorem~{\rm\ref{thm_criterion}}, consider the restriction of the Markov process \eqref{Markovchain} to $Y$, denoted by $(y_n,\mathbb{P}_y)$. Then there exists a unique invariant measure $\mu\in \mathcal{P}(Y)$. Moreover, there exist constants $C,\gamma>0$ such that
    \begin{equation}\label{EM-2}
        \|\mathscr{D}(y_n)-\mu\|_L^*\le Ce^{-\gamma n}\quad \text{for any }y_0\in Y,\ n\in \mathbb{N}.
    \end{equation}
\end{proposition}

Indeed, the derivation of Theorem~\ref{thm_criterion} from Proposition~\ref{prop_restricY} and hypothesis $(\mathbf{H1})$ is verbatim as \cite[Proposition 2.4]{LWXZZ-24}, and is hence omitted for the sake of brevity.\footnote{We mention that the invariant measures $\mu$ in \eqref{EM-1} and \eqref{EM-2} are the same, while the exponential rates $\gamma>0$ may differ. The careful reader might find the notion $\|\cdot\|_L^*$ in \eqref{EM-2} ambiguous without specifying whether this is the dual metric of $L_b(X)$ or $L(Y)$. In fact, the two of them give rise to the same distance on $\mathcal{P}(Y)$, thanks to McShane’s lemma that any Lipschitz function on $Y$ admits an extension to $X$ with Lipschitz norm preserved.} 

\medskip

In the rest of this section, we concentrate on the compact invariant subset $Y$ and the restricted Markov chain $(y_n,\mathbb{P}_y)$. Since $Y$ is compact, the standard Krylov--Bogolyubov method (see, e.g., \cite[Section 2.5]{KS-12}) yields the existence of invariant measures. The non-trivial parts of Proposition~\ref{prop_restricY} lie in the uniqueness and mixing property \eqref{EM-2}. The proof consists of two steps:

\begin{itemize}
    \item In Section~\ref{subsec_approximatectrl}, we exploit the new asymptotic compactness hypothesis $(\mathbf{H2})$ to establish a control result (Lemma~\ref{lem_approxcontrol}). This observation is the core novelty of our criterion.

    \item In Section~\ref{subsec_reduction}, we adapt and refine techniques from \cite{KNS-20,NZZ-24} to derive exponential mixing on the compact subset $Y$ from control results.
\end{itemize}
Before carrying out all the details, let us outline the strategy of proof.

\subsection{Overview of proof}\label{subsec_overviewRDS}

We begin by recalling a widely-used strategy for establishing mixing in the presence of a smoothing effect, which typically occurs in parabolic PDEs. Then we highlight the difficulties when attempting to adapt this approach to dispersive equations, and illustrate the key ideas that allow us to overcome them.

\subsubsection{Widely-used strategy for parabolic systems}

Since the works of Doeblin and Harris, it has been understood that mixing for Markov processes can be deduced from two properties: 

\begin{itemize}
    \item Irreducibility, typically a consequence of the system's dissipation mechanism; and

    \item A coupling condition, ensuring that trajectories draw closer in a probabilistic sense.
\end{itemize}
Concerning random PDEs, global dissipation yields irreducibility, while the coupling condition often follows from control theory. This paradigm can be summarized as:
\[\text{control }\Rightarrow \text{ coupling }\Rightarrow\text{ mixing}\]

Taking the 2D Navier--Stokes system for example, Hairer--Mattingly \cite{HM-06,HM-08,HM-11} exploited hypoellipticity and Malliavin calculus when the noise is white-in-time, Shirikyan \cite{Shi-15,Shi-21} clarified the connection between mixing and control when studying spatially localized or boundary noises, and Kuksin--Nersesyan--Shirikyan \cite{KNS-20} extends the geometric control method to handle noises of random Haar series. Our present paper is particularly inspired by \cite{KNS-20}, which paves the way from a control result to the mixing on compact space $Y$; see Section~\ref{subsec_reduction}.

% 和文献综述有点重复

\subsubsection{Challenges for hyperbolic equations: lack of regularization}

As observed in \cite{LWXZZ-24,CXZZ-25,CLXZ-24}, the lack of smoothing effect causes essential difficulties in many aspects. Two major downsides are:

\begin{itemize}
    \item[(1)] Many probabilistic methods rely on the compactness of the phase space. Moreover, the control property often requires extra regularity (see, e.g., \cite[Section 4]{CXZZ-25}). The compensations for compactness and regularity are not evident without smoothing effect.
\end{itemize}

This first difficulty can be addressed by introducing the concept of EAC from \cite{LWXZZ-24}. In this work, we harness the EAC hypothesis $(\mathbf{H1})$ for a compactness reduction from Theorem~\ref{thm_criterion} to Proposition~\ref{prop_restricY}. Nevertheless, the second obstacle is even more subtle:

\begin{itemize}
    \item[(2)]  Although we assumed the approximate controllability $(\mathbf{H4})$, quantitative control estimates are needed to imply the desired coupling condition. To this end, \cite{KNS-20} exploited the smoothing effect of the linearized operator, which fails in our setting.
\end{itemize}

Despite the lack of a direct derivative gain, we find it possible to combine the asymptotic compactness hypothesis $(\mathbf{H2})$ with system evolution, which is the central insight of Section~\ref{subsec_approximatectrl}. This idea is reminiscent of the Nash--Moser iteration, often invoked to handle derivative loss. Roughly speaking, the goal of iteration is to construct compact approximations to the derivative-losing operators and manage the error terms effectively. Typically, one either uses refining approximations at each step, or shows that the error term is a contraction so that the total error forms a summable geometric series. Clearly, the latter idea is reflected by \eqref{ACL-2}.

\subsubsection{Analogy to Nash--Moser iteration}

Before ending this overview, we also take the chance to discuss in more detail the control property, which is a crucial step to exponential mixing, as well as the new observation that links our approach to iteration.

The mixing property can be reduced (via several nontrivial implications) to the following:

\begin{lemma}\label{lem_approxcontrol}
    Under the assumptions of Proposition~{\rm\ref{prop_restricY}}, let $q_1\in (q_0,1)$ be arbitrarily given. Then for any $\sigma\in (0,1)$, there exist constants $\delta\in (0,1)$ and $C_\sigma>0$, a Borel map $\Phi\colon Y\times X\times E\to E$\footnote{In the sequel we often write $\Phi^{y,x}=\Phi(y,x,\cdot)\colon E\to E$ and view $x,y$ as parameters.}, and a family of Borel subsets $K^{y,\sigma}\subset K^y$ for each $y\in Y$, such that
    \begin{equation}\label{ApproxCtrl-1}
        \ell(K^{y,\sigma})\ge 1-\sigma\quad \text{for any }y\in Y,
    \end{equation}\begin{equation}\label{ApproxCtrl-2}
        \|\ell-(\Id+\Phi^{y,x})_*\ell\|_{TV}\le C_\sigma\|y-x\|^{1/2}\quad \text{for any }y\in Y,\ x\in X.
    \end{equation}
    Moreover, if $\zeta\not \in K^{y,\sigma}$, then $\Phi^{y,x}(\zeta)=0$; and if $\zeta\in K^{y,\sigma}$, then
    \begin{equation}\label{ApproxCtrl-3}
        \|S(y,\zeta)-S(x,\zeta+\Phi^{y,x}(\zeta))\|\le q_1 \|y-x\|\quad \text{whenever }\|y-x\|\le \delta.
    \end{equation}
\end{lemma}

Among others, the key assertion is \eqref{ApproxCtrl-3}, which means two trajectories can be driven closer by a suitable control of the form $\xi=\zeta+\Phi^{y,x}(\zeta)$. This is referred to as ``local stabilization" in the terminology of control theory.
% As the noise is degenerate, exact controllability typically fails; that is, we cannot not generally find $\xi\in E$ so that
% \[\|S(y,\zeta)-S(x,\xi)\|=0.\]
Heuristically, in order to fulfill \eqref{ApproxCtrl-3}, we intend to cancel the first-order terms in the Taylor expansion,
\[S(y,\zeta)-S(x,\xi)=-D_y S(y,\zeta)(x-y)-D_\zeta S(y,\zeta)(\xi-\zeta)+(\text{higher order terms}).\]
This suggests a candidate for $\xi$ as:
\begin{equation}\label{Phichoice}
    \text{ `` }\xi=\zeta-(D_\zeta S(y,\zeta))^{-1}D_y S(y,\zeta)(x-y)\text{ "}.
\end{equation}
As there is no sign for $D_\zeta S(y,\zeta)$ to be invertible, the inverse map $(D_{\zeta} S(y,\zeta))^{-1}$ on the right-hand side is meaningless. In practice, previous works such as \cite{HM-06,KNS-20} replace the inverse with the Moore--Penrose pseudo-inverse, which is defined for any bounded operator $A\colon E\to X$ as
\begin{equation}\label{Rgammadef}
    R^\gamma=A^*(AA^*+\gamma )^{-1}\colon X\to E.
\end{equation}
Here $\gamma>0$ is a small parameter. If $A$ has dense image, then by standard functional analysis,
\[\lim_{\gamma\to 0+} \|AR^\gamma x-x\|=0\quad \text{for any }x\in X.\]
Therefore, $R^\gamma$ is indeed an approximated right inverse of $A$

Nevertheless, one cannot expect this convergence to hold uniformly, i.e.~with respect to operator norm. As a result, quantitative estimate is out of scope. One possible remedy comes from compactness. Indeed, as $V\subset X$ is compactly embedded, it is not hard to show that (see, e.g., \cite[Proposition 2.6]{KNS-20}), for any $\varepsilon>0$ there exists a small parameter $\gamma>0$ so that
\begin{equation}\label{pseudo-inverse-V}
    \|AR^\gamma v-v\|_X\le \varepsilon\|v\|_V\quad \text{for any }v\in V.
\end{equation}
In parabolic settings, the smoothing effect yields $D_y S(y,\zeta)(x-y)\in V$, making \eqref{pseudo-inverse-V} applicable.

\medskip

In contrast, for hyperbolic equations, the linearized evolution $D_u S(y,\zeta)(x-y)$ merely belongs to $X$ rather than $V$. This mirrors the derivatives-loss issue in functional equations, which we now briefly illustrate. When solving an abstract functional equation $F(u)=0$, Newton's iteration method considers the recurrence sequence
\[DF(u_n)(u_{n+1}-u_n)=-F(u_n).\]
In other words, as long as $DF(u_n)$ is invertible, we set (cf.~\eqref{Phichoice})
\begin{equation}\label{Newtoniteration}
    u_{n+1}=u_n-DF(u_n)^{-1}(F(u_n)).
\end{equation}
The problem occurs when $(DF)^{-1}\circ F$ does not pertain the regularity, and hence one cannot find a suitable function space to carry out this recursion. 

The Nash--Moser iteration remedies this by introducing smoothing approximations, so that solutions converge despite derivative loss. For example, one can replace $DF^{-1}\circ F$ by Fourier truncation to the first $M_n$ modes at the $n$-th step, and let $M_n$ increase to $\infty$ in an appropriate rate, to ensure that $u_n$ converges to the solution of $F(u)=0$. Another favorable situation is when $DF^{-1}$ can be decomposed into $K+C$, where $K$ is compact so that its approximate inverse gains extra regularity, and $C$ is a contraction map containing small errors.

\medskip

Returning to our problem, we need to improve the intuitive choice \eqref{Phichoice}, where the range of operator $D_y S(y,\zeta)$ merely belongs to $X$ rather than $V$. To this end, we use the compact operator $T(y,\zeta)\in \mathcal{L}(X,V)$ in place of $D_y S(y,\zeta)$. The error can be estimated by virtue of \eqref{ACL-2}. Specifically, we exploit a decomposition
\[D_y S(y,\zeta)=(\text{compact)}+\text{(contraction)}.\]
Then the revised version of \eqref{Phichoice}, reading (recall $R^\gamma$ refers to a pseudo-inverse of $D_\zeta S(y,\zeta)$)
\begin{equation}\label{Phichoice-revised}
    \text{ `` }\xi=\zeta-R^\gamma(y,\zeta)T(y,\zeta)(x-y)\text{ "}
\end{equation}
will serve to imply local stabilization \eqref{ApproxCtrl-3}.

\subsection{Control property}\label{subsec_approximatectrl}

To better highlight our contributions and new ideas, we extract the local stabilization \eqref{ApproxCtrl-3} from other assertions of Lemma~\ref{lem_approxcontrol} in the following lemma, which is of independent interest. The genuine proof of Lemma~\ref{lem_approxcontrol} is presented in Appendix~\ref{appendix_meastrans}, presuming some results from \cite{KNS-20, NZZ-24} that are less relevant to our current discussion.

We mention that this lemma involves purely deterministic dynamics, and can be applied to study the local stabilization of PDE models (cf.~\cite{Sar-12} for cubic NLS).

\begin{lemma}\label{lem_contoltoy}
    Under the assumptions of Proposition~{\rm\ref{prop_restricY}}, for any $q_1\in (q_0,1)$, $y\in Y$ and $\zeta\in K^y$ (see \eqref{K^ydef}), there exist constants $C>0$ and $\delta \in (0,1)$ (depending on $y$ and $\zeta$), such that if $x\in X$ satisfies $\|y-x\|\le \delta$, then one can find $\xi\in E$ so that
    \begin{equation}\label{ApproxCtrl-toy-1}
        \|\zeta-\xi\|_E\le C\|y-x\|,
    \end{equation}
    \begin{equation}\label{ApproxCtrl-toy-2}
        \|S(y,\zeta)-S(x,\xi)\|\le q_1 \|y-x\|.
    \end{equation}
\end{lemma}

\begin{proof}
    Fix $y\in Y$ and $\zeta\in K^y$ from now on. Let us denote $A:=D_\zeta S(y,\zeta)$ and $T=T(y,\zeta)$ (see $(\mathbf{H2})$). Following the guideline from last subsection, we intend to define $\xi$ as
    \[\xi=\zeta-R^\gamma T(x-y),\]
    where $R^\gamma$ is the pseudo-inverse of $A$ given by \eqref{Rgammadef}, and $\gamma\in (0,1)$ is determined in the following manner: Choose an auxiliary parameter $\varepsilon\in (0,1)$ so that (recall $C_0$ comes from \eqref{ACL-1})
    \begin{equation}\label{q0+Cepsilon<q1}
        q_0+C_0\varepsilon<q_1.
    \end{equation}
    Then the hypothesis $(\mathbf{H4})$ (i.e., $A$ has dense image) and \cite[Proposition 2.6]{KNS-20} yields a sufficiently small $\gamma$ to ensure \eqref{pseudo-inverse-V} with this specific $\varepsilon$. We emphasis that $\gamma$ is determined by $y$ and $\zeta$.

    First we address \eqref{ApproxCtrl-toy-1}. It is easy to see from \eqref{Rgammadef} that $\|R_\gamma\|_{\mathcal{L}(X,E)}\le \gamma^{-1}\|A\|_{\mathcal{L}(E,X)}$. Thus
    \[\|\zeta-\xi\|_E\le \gamma^{-1} \|A\|_{\mathcal{L}(E,X)}\|T\|_{\mathcal{L}(X)}\|y-x\|\le C_{y,\zeta}\|y-x\|.\]
    Here and later, the constant $C_{y,\zeta}$ depends on $y,\zeta$ but not $x$.
    
    Next we turn to \eqref{ApproxCtrl-toy-2}, provided $\|y-x\|\le \delta$, with $0<\delta\ll 1$ to be chosen later. By virtue of Taylor's expansion and our assumption on the local boundedness of second-order derivatives of $S$, we find (with the constant $C$ in the first line independent of $x,y,\eta$)
    \begin{align*}
        \|S(y,\zeta)-S(x,\xi)\|
        \le &\|-D_y S(y,\zeta)(x-y)-D_\zeta S(y,\zeta)(\xi-\zeta)\|+C (\|y-x\|^2+\|\xi-\zeta\|_E^2)\\
        \le &\|-D_y S(y,\zeta) (x-y)+AR^\gamma T(x-y)\|+C_{y,\zeta}\|y-x\|^2\\
        \le & \|T(x-y)-D_y S(y,\zeta)(x-y)\|+\|T(x-y)-AR^\gamma T(x-y)\|+C_{y,\zeta}\|y-x\|^2
    \end{align*}
    The first term on the right-hand side can be estimated by \eqref{ACL-2}:
    \[\|T(x-y)-D_yS(y,\eta)(x-y)\|\le q_0 \|y-x\|;\]
    while the second term can be estimated via \eqref{ACL-1} and \eqref{pseudo-inverse-V}:
    \[\|T(x-y)-AR^\gamma T(x-y)\|_X\le \varepsilon \|T(x-y)\|_V\le C_0 \varepsilon \|y-x\|_X.\]
    Therefore, taking $\|y-x\|\le \delta$ into account, we find
    \begin{equation}\label{stabilization-1}
        \|S(y,\zeta)-S(x,\xi)\|\le (q_0+C_0 \varepsilon +C_{y,\zeta} \delta) \|y-x\|.
    \end{equation}
    Thanks to \eqref{q0+Cepsilon<q1}, we obtain \eqref{ApproxCtrl-toy-2} once $\delta$ is sufficiently small (depending on $y$ and $\zeta$).
\end{proof}

We remark that the proof of Lemma~\ref{lem_approxcontrol} follows from similar lines, while the main difference is that the operator $A(y,\zeta):=D_\zeta S(y,\zeta)$ is now parameter-dependent. Thus the pseudo-inverse $R^\gamma(y,\zeta)$ is only a nice approximation of right inverse ``at most points". Indeed, as realized by \cite{KNS-20,NZZ-24}, given any $\sigma\in (0,1)$, for $\zeta$ belonging to the appropriately defined Borel set $K^{y,\sigma}\subset K^y$, whose $\ell$-measure is at least $1-\sigma$, the parameter $\gamma$ for the approximate inverse $R^{\gamma}(y,\zeta)$, as well as various constants (especially $C_{y,\zeta}$ in \eqref{stabilization-1}), can be chosen to be uniform in $y,\zeta$ and depending only on $\sigma$. The interested reader is referred to Appendix~\ref{appendix_meastrans} for technical details.

\subsection{Mixing on $Y$ via control and coupling}\label{subsec_reduction}

We follow the route in \cite{KNS-20} for the implication from control to mixing on $Y$. As outlined earlier, we first deduce coupling from control, and then establish mixing from coupling. While the ideas are not new, the technical details differ and require careful modifications. To keep the focus on the novel aspects of this work, we will only sketch the coupling condition and refer to Appendices \ref{appendix_control->coupling}–\ref{appendix_coupling->mixing} for complete proofs.

\medskip

\noindent\textbf{Coupling condition.} Let $q_2\in (q_0,1)$ be arbitrarily given. Since $S$ is locally Lipschitz, there exists an integer $L=L(q_2)\in \mathbb{N}$ such that
\begin{equation}\label{LipS}
    \|S(y,\zeta)-S(y',\zeta)\|\le q_2^{-L}\|y-y'\|\quad \text{for any }y,y'\in Y,\ \zeta\in K.
\end{equation}
Let $R>0$ be a number so that $Y\subset B_X(R)$. For any $d\in (0,1)$, let us decompose $\mathbf{Y}=Y\times Y$ in the following manner: pick $N=N(d)$ as the least positive integer so that $q_2^N R\le d/2$, and introduce for $n\ge 0$ and $-N\le k\le -1$ the Borel subsets (here $a\vee b:=\max\{a,b\}$)
\begin{align*}
    \mathbf{Y}_\infty&=\{(y,y')\in \mathbf{Y}:y=y'\},\\
    \mathbf{Y}_n&=\{(y,y')\in \mathbf{Y}:q_2^{n+1}d<\|y-y'\|\le q_2^n d\},\\
    \mathbf{Y}_k&=\{(y,y')\in \mathbf{Y}:\|y-y'\|>d,\  Rq_2^{N+k+1}<\|y-\tilde{y}\|\vee \|y'-\tilde{y}\|\le Rq_2^{N+k}\}.
\end{align*}
Here in the definition of $\mathbf{Y}_k$, the point $\tilde{y}$ is the ``stationary state" appearing in hypothesis $(\mathbf{H3})$. One readily checks that $\mathbf{Y}$ is the disjoint union of these subsets $\mathbf{Y}_l\, (-N\le l\le \infty)$. We also employ the self-explanatory notations such as $\mathbf{Y}_{\ge l}=\bigcup_{l\le j\le \infty} \mathbf{Y}_j$ and $\mathbf{Y}_{<l}=\bigcup_{-N\le j<l} \mathbf{Y}_j$.

\medskip

The coupling result connecting control property and exponential mixing is stated as follows.

\begin{lemma}\label{lem_coupling}
    Under the assumptions of Proposition~{\rm\ref{prop_restricY}}, let $q_2\in (q_0,1)$ be arbitrarily given. Then for any $\nu\in (0,1)$, there exist constants $C_\nu>0$ and $d_\nu\in (0,1)$, such that for any $d\in (0,d_\nu]$, there are measurable maps $V,V'\colon Y\times Y\times \Omega\to Y$ satisfying the following assertions.

    \begin{itemize}
        \item[(a)] For $(y,y')\in \mathbf{Y}$, the pair $(V(y,y'),V'(y,y'))$ is a coupling between $P(y,\cdot)$ and $P(y',\cdot)$. Moreover, if $(y,y')\in \mathbf{Y}_\infty\cup \mathbf{Y}_{<0}$ (i.e.~$y=y'$ or $\|y-y'\|>d$), then
        \[V(y,y')=S(y,\tilde{\eta})\quad \text{and}\quad V'(y,y')=S(y',\tilde{\eta}),\]
        where $\tilde{\eta}$ is an i.i.d.~copy of the random variables $\eta_n$ in the random dynamical system.

        \item[(b)] If $(y,y')\in \mathbf{Y}_n$ for some $0\le n<\infty$, then
        \begin{equation}\label{coupling-1}
            \mathbb{P}((V(y,y'),V'(y,y'))\in  \mathbf{Y}_{\ge n+1})\ge 1-\nu.
        \end{equation}
        Moreover, if $(y,y')\in \mathbf{Y}_n$ with $L\le n<\infty$, then
        \begin{equation}\label{coupling-2}
            \mathbb{P}((V(y,y'),V'(y,y'))\in  \mathbf{Y}_{<n-L})\le C_\nu\|y-y'\|^{1/2}.
        \end{equation}
    \end{itemize}
\end{lemma}

\medskip

Now we can conclude the proof of Proposition~\ref{prop_restricY}, namely the exponential mixing on $Y$.

\begin{proof}[\rm \bf Proof of Proposition~\ref{prop_restricY}]

The control property Lemma~\ref{lem_approxcontrol} is proved in Appendix~\ref{appendix_meastrans}, incorporating some observations from the last subsection. According to an idea similar to \cite[Section 2.3]{KNS-20}, we can derive from it the coupling condition Lemma~\ref{lem_coupling}; see Appendix~\ref{appendix_control->coupling} for details.

Finally, the derivation of Proposition~\ref{prop_restricY} from Lemma~\ref{lem_coupling} is similar to \cite[Section 2.2]{KNS-20}. Since \cite{KNS-20} only treated the case $m_0=1$ in $(\mathbf{H3})$, for the sake of completeness, we provide the rigorous proof in Appendix~\ref{appendix_coupling->mixing}. Now the proof of Proposition~\ref{prop_restricY} is complete.
\end{proof}

As mentioned earlier, the exponential mixing on $X$ follows from Proposition~\ref{prop_restricY} by exploiting the EAC hypothesis $(\mathbf{H1})$; see \cite[Proposition 2.4]{LWXZZ-24} for details. Therefore, we have now accomplished the proof of Theorem~\ref{thm_criterion}, namely the abstract criterion has been established.

\section{Global dynamics of Schr\"odinger equations}\label{sec_globaldynamic}

Before applying our general criterion to prove the Main Theorem (see Section~\ref{sec_applyNLS}), we also need to investigate the deterministic NLS equation. Our focus in this section is on the global dynamics, especially the asymptotic compactness, which guarantee hypotheses $(\mathbf{H1})$--$(\mathbf{H3})$. The geometric control property, corresponding to $(\mathbf{H4})$, will be addressed separately in Section~\ref{sec_geometriccontrol}.

We first review the deterministic version of NLS equation, which reads
\begin{equation}\label{deterministicNLS}
\left\{\begin{array}{ll}
iu_t+u_{xx}+ia(x)u=|u|^{p-1}u+f(t,x),\\
u(0,\cdot)=u_0\in H^1(\mathbb{T}),
\end{array}\right.
\end{equation}
where $p\ge 3$ is an odd integer, and $f\colon [0,1]\rightarrow H^1(\mathbb{T})$ (or $f\colon \mathbb{R}^+\rightarrow H^1(\mathbb{T})$) is a deterministic force. The well-posedness and global stability are revisit in Section~\ref{sec_NLSprerequist}. The nonlinear smoothing effect, introduced by Bourgain \cite{Bourgain-98}, is discussed in Section~\ref{subsubsec_nonlinearsmoothing}

Then we analyze the linearized system around the reference trajectory $u\colon [0,1]\to \mathbb{C}$, which plays a central role in verifying hypothesis $(\mathbf{H2})$:
\begin{equation}\label{D_uS}
    \left\{\begin{array}{ll}
iv_t+v_{xx}+ia(x) v=\frac{p+1}{2} |u|^{p-1} v+\frac{p-1}{2} |u|^{p-3} u^2 \bar{v},\\
v(0,\cdot)=v_0.
\end{array}\right.
\end{equation}
The key result of this section concerning asymptotic compactness is stated as follows:

\begin{proposition}\label{prop_H2}
    Let $b\in (1/2,1)$ and $R>0$ be arbitrarily given. There exists a constant $C>0$, such that if the reference trajectory $u\in X_1^{5/4,b}$ (Bourgain spaces, see Definition~{\rm\ref{def_Bourgain}}) with $\|u\|_{X_1^{5/4,b}}\le R$, then for any $v_0\in H^1(\mathbb{T})$, the solution $v\in C(0,1;H^1(\mathbb{T}))$ of the linearized equation \eqref{D_uS} satisfies
    \begin{equation}\label{H2-1}
        \|v(1)-e^{-i\theta_u(1)} S_a(1)v_0\|_{H^{5/4}}\le C\|v_0\|_{H^1}.
    \end{equation}
    Here $S_a(t)$ is the $C_0$-operator group generated by $i\partial_x^2-a(x)$, and $\theta_u(t)\in \mathbb{R}$ is defined by
    \begin{equation}\label{theta(u,t)def}
        \theta_u(t)=\frac{p+1}{4\pi}\int_0^t \|u(s)\|_{L^{p-1}(\mathbb{T})}^{p-1}\, ds.
    \end{equation}
\end{proposition}

The proof constitutes Section~\ref{subsubsec_H2}, relying on multilinear estimates in Bourgain spaces. The merit is that, although the solution $v(t)$ and linear evolution $S_a(t)v_0$ retain the same regularity $H^1$, their difference, up to a ``phase shift" $e^{-i\theta_u(t)}$, gains extra regularity $H^{5/4}$.

\subsection{Preliminaries: well-posedness and global stability}\label{sec_NLSprerequist}

We quickly review the deterministic NLS equation. The results in this subsection are taken from \cite{CXZZ-25}.

\subsubsection{Well-posedness and smoothness of the solution map}

By a solution of NLS equation \eqref{deterministicNLS}, we always mean a mild solution. More precisely, given $u_0\in H^s(\mathbb{T})$, the solution is formulated as $u\in C(0,1;H^s(\mathbb{T}))$ satisfying the Duhamel formula
\[u(t)=S_a(t)u_0-i\int_0^t S_a(t-s)(|u(s)|^{p-1} u(s)+f(s))\, ds.\]
As $H^s(\mathbb{T})$ is a Banach algebra for $s>1/2$, the standard fixed-point argument yields local well-posedness. Using $H^1$-energy balance and Sobolev embedding $H^1(\mathbb{T})\hookrightarrow L^\infty(\mathbb{T})$, it is easy to derive global well-posedness in $H^s(\mathbb{T})$ for any $s\ge 1$.

\medskip

For future convenience, we need a finer result that, the solution $u$ actually belongs to the Bourgain space $X_1^{s,b}$ for any $b\in (1/2,1)$, where the subscript $1$ refers to the time interval $[0,1]$.

\begin{definition}[Bourgain spaces]\label{def_Bourgain}
    Let $s,b\in \mathbb{R}$ be arbitrarily given. The Bourgain space $X^{s,b}$ consists of functions $u:\mathbb{R}\times\mathbb{T}\rightarrow\mathbb C$ for which the norm defined by
    \[\|u\|_{X^{s,b}}:=\left(\sum_{k\in\mathbb{Z}}\int_\mathbb{R}\langle k\rangle^{2s}\langle \tau+k^2\rangle^{2b}|\widehat u(\tau,k)|^2d\tau\right)^{1/2}<\infty.\]
    Here $\hat{u}(\tau,k)$ refers to the space-time Fourier transform of $u$. For $T>0$, the restricted space $X^{s,b}_T$ consists of $u\colon [0,T]\times \mathbb{T}\to \mathbb{C}$ with norm
    \[\|u\|_{X^{s,b}_T}=\inf\{\|\tilde u\|_{X^{s,b}}:\tilde u=u\ {\rm on\ }[0,T]\times\mathbb{T}\}<\infty.\]
\end{definition}

Bourgain spaces were originally used for studying low-regularity well-posedness \cite{Bourgain-93}, and then grow into a powerful tool for dispersive PDEs. We refer the reader to \cite[Appendix A.1]{CXZZ-25} for some basic facts (some properties in need of this paper is recalled in Appendix~\ref{appendix_Bourgain}), and \cite[Appendix A.2]{CXZZ-25} for a sketched proof of the following well-posedness result:

\begin{lemma}\label{lem_wp}
    Let $s\geq 1$ and $b\in (1/2,1)$ be arbitrarily given. Then for every $u_0\in H^s(\mathbb{T})$ and $f\in L^2(0,1;H^s(\mathbb{T}))$, the NLS equation \eqref{deterministicNLS} admits a unique solution $u\in X^{s,b}_1$.
\end{lemma}

Due to Lemma~\ref{lemBourgain_embed}, for any $b>1/2$, the Bourgain space $X^{s,b}_T$ is continuously embedded into $C(0,T;H^s)$. Thus the uniqueness implies that the mild solution actually belongs to $X_1^{s,b}$.

\medskip

From now on we define the time-$1$ solution map of \eqref{deterministicNLS} as
\begin{align*}
    S\colon H^1(\mathbb{T})\times L^2(0,1;H^1(\mathbb{T}))&\to H^1(\mathbb{T}),\\
    (u_0,f)&\mapsto u(1).
\end{align*}
The map $S$ is differentiable (see, e.g., \cite[Corollary 5.6]{LP-15}). More precisely:

\begin{itemize}
    \item[\tiny$\bullet$] For $v_0\in H^1(\mathbb{T})$, it turns out that $D_{u_0} S(u_0,f) (v_0)$ is the time-$1$ solution of linearized equation \eqref{D_uS}, where $u$ solves the NLS \eqref{deterministicNLS} for $t\in [0,1]$. We denote this map by
\[D_{u_0} S(u_0,f) (v_0)=\mathcal{R}_u(1,0) v_0.\]
More generally, we also define the solution map of \eqref{D_uS} from time $s$ to $t$ by $\mathcal{R}_u(t,s)$.

    \item[\tiny$\bullet$] Similarly, if $g\in L^2(0,1;H^1(\mathbb{T}))$, then $D_f S(u_0,f)(g)$ is the time-$1$ solution of
    \begin{equation}\label{D_etaS}
    \left\{\begin{array}{ll}
iv_t+v_{xx}+ia(x) v=\frac{p+1}{2} |u|^{p-1} v+\frac{p-1}{2} |u|^{p-3} u^2 \bar{v}+g,\\
v(0,x)=0.
\end{array}\right.
\end{equation}
According to the Duhamel formula, we have
\begin{equation}\label{ADuhamel}
    D_f S(u_0,f) (g)=\int_0^1 \mathcal{R}_u (1,t) g(t)\, dt.
\end{equation}
\end{itemize}

The well-posedness of linear systems \eqref{D_uS} and \eqref{D_etaS} are readily established (see Lemma~\ref{lem_wplinearization}(a)), which yields the local boundedness of first-order derivatives of $S$. Similarly, higher-order derivatives of $S$ are nothing but higher-order linearizations of the NLS equation. And since the nonlinear term $|u|^{p-1} u$ is smooth in $u$, it is easy to prove that:

\begin{lemma}\label{lem_Ssmooth}
    The time-$1$ solution map $S\colon H^1\times L^2(0,1;H^1)\to H^1$ of the NLS equation \eqref{deterministicNLS} is smooth, and its derivatives of any order are locally bounded.
\end{lemma}

\subsubsection{Global stability}

The dissipativity hypothesis $(\mathbf{H3})$ is often a consequence of global stability, which means in the absence of external force, the trajectory decays to $0$ exponentially. As explained in the introduction, the difficulty lies in the localized structure of damping, and can be tackled by Carleman estimates. Specifically, \cite{LM-23} proved the following result for cubic NLS, and then \cite{CXZZ-25} extends it to nonlinearity of any odd order $p\ge 3$:

\begin{proposition}[{\cite[Proposition 2.2]{CXZZ-25}}]\label{prop_stability}
    There exist constants $\beta,C>0$, such that
    \[E(u(t))\leq Ce^{-\beta t} E(u_0)\quad  \text{for any } u_0\in H^1(\mathbb{T}),\ t\geq 0,\]
    where $u(t)$ stands for the solution of NLS equation \eqref{deterministicNLS} with $f(t,x)\equiv 0$.
\end{proposition}

\medskip

For later convenience, we also mention the global stability for linear evolutions, namely the exponential decay of the operator semigroup $S_a(t)$. According to \cite[Proposition 4.1]{RZ-09}, we have 
\[\|S_a(t)\|_{\mathcal{L}(H^1)}\le Ce^{-\beta t}\]
Then \cite{CXZZ-25} realized that the constant $C$ can be chosen to be $1$, up to an equivalent norm. The benefit is that $S_a(t)$ is immediately contractive for any $t>0$.
% In fact, according to \cite[Proposition 4.1]{RZ-09}, there exists a constant $\beta>0$, such that for any $s\ge 0$, one can find $C>0$ satisfying
% \[\|S_a(t)\|_{\mathcal{L}(H^1)}\le Ce^{-\beta t}\quad \text{for any }t\ge 0.\]
% Invoking a trick in \cite{Paz-83}, the constant $C$ can be taken to be $1$ if we replace the norm by an equivalent one. We shall denote this equivalent space by $\tilde{H}^1$ from now on.

\begin{lemma}[{\cite[Lemma 4.10]{CXZZ-25}}]\label{lem_equivalentnorm}
    There exists an equivalent norm $\|\cdot\|_{\tilde{H}^1}$ on $H^1(\mathbb{T})$, and a constant $\beta>0$, so that
    \[\|S_a(t)\|_{\mathcal{L}(\tilde{H}^1)}\le e^{-\beta t}\quad \text{for any }t\ge 0.\]
\end{lemma}

% Indeed, this equivalent norm can be implicitly defined by (see \cite[Chapter 1, Theorem 5.2]{Paz-83})
% \[\|u_0\|_{\tilde{H}^1}:=\sup_{t\ge 0} \|S_a(t)u_0\|_{H^1}.\]
% Alternatively, in \cite{CXZZ-25} the authors found an explicit definition that when $s=2n$, one can choose
% \[\|u\|_{\tilde{H}^1}=\sum_{k=0}^n A^{n-k}\|(i\partial_x^2-a(x))^k u_0\|_{L^2},\quad \text{where }A\gg 1.\]

\subsection{Asymptotic compactness via nonlinear smoothing}\label{subsec_ACviaNS}

We first explain the idea of nonlinear smoothing, and recall the EAC for NLS equation from \cite{CXZZ-25}. Then we carry out the proof of Proposition~\ref{prop_H2} via multilinear estimates in Bourgain spaces.

\subsubsection{Nonlinear smoothing}\label{subsubsec_nonlinearsmoothing}

% Bourgain's original discovery in \cite{Bourgain-98} of nonlinear smoothing effect is that, for the solution $u(t)$ of NLS equation on $\mathbb{R}^2$, the nonlinear part $u(t)-S(t)u_0$ gains extra regularity. Then Keraani and Vargas \cite{KV-09} prove nonlinear smoothing on $\mathbb{R}^d$. When extending to NLS equation on torus $\mathbb{T}$, similar results are proved by \cite{ET-13,McConnell-22} with an additional phase shift (cf.~the shift term $e^{-i\theta_u(1)}$ in \eqref{H2-1}). The phase shift cancels out the ``single resonance", which is an essential obstacle that prevents nonlinear smoothing; see, e.g., \cite[Section 4.1]{ET-16}.

We follow \cite{CXZZ-25,ET-13,McConnell-22} to introduce the terminology of resonance decomposition. For functions $u_1,\dots ,u_p\colon \mathbb{T}\to \mathbb{C}$, define the $p$-multiplication operator
\begin{equation}\label{p_product}
    \mathcal{N}(u_1,\dots,u_p)=\prod_{l\, {\rm odd}}u_l \prod_{l\, {\rm even}}\bar u_l.
\end{equation}
Written in Fourier coefficients, for each $k\in \mathbb{Z}$ we have
\[\mathcal{F} \mathcal{N}(u_1,\dots ,u_p)(k)=(2\pi)^{-(p-1)/2}\sum_{k=k_1-k_2+\cdots+k_p}\prod_{l\, {\rm odd}}\widehat u_l(k_l) \prod_{l\, {\rm even}} \overline{\widehat u_l(k_l)}\]
A configuration of frequencies $(k_1,\dots,k_p)$ with $k_l\in\mathbb{Z}$ and $k=k_1-k_2+\cdots +k_p$ is called {\it resonant}, if there is an odd $m$ such that 
$k=k_m$. Define an auxiliary $p$-linear form $\mathcal{N}_{R}$ as
\begin{equation}\label{resonantfrequency}
\mathcal{F}\mathcal{N}_R(u_1,\dots,u_p)(k)=(2\pi)^{-(p-1)/2}\sum_{\substack{m=1\\{\rm odd}}}^{p}\sum_{\substack{ k=k_1-k_2+\cdots+k_p\\k=k_m}}\prod_{l\, {\rm odd}}\widehat u_l(k_l) \prod_{l\, {\rm even}} \overline{\widehat u_l(k_l)}.
\end{equation}
We mention that the single-resonant (i.e.~$k=k_m$ for exactly one odd $m$) terms appear in $\mathcal{N}_R$ exactly once, while other resonant terms may occur repeatedly. Then the difference 
\[\mathcal{N}_{NR}(u_1,\dots ,u_p):=\mathcal{N}(u_1,\dots ,u_p)-\mathcal{N}_R(u_1,\dots ,u_p)\] 
carries no single-resonance. The essence of nonlinear smoothing is that $\mathcal{N}_{NR}$ enjoys extra regularity. More precisely, we have the following multilinear estimate (see also \cite{ET-13,McConnell-22}). The difficulty lies in the $X^{s+\sigma,-b'}$ (rather than $X^{s,-b'}$) regularity of $\mathcal{N}_{NR}$.

\begin{lemma}[{\cite[Lemma 2.9]{CXZZ-25}}]\label{lem_nonlinearsmoothing}
    Let $T>0,s\geq 1,b>1/2$ and $\sigma\in(0,1/4]$ be arbitrarily given. Then for every $b'\in [\sigma,1/2)$, there exists a constant $C>0$ such that for any $u_1,\cdots,u_p\in X_T^{s,b}$,
    \begin{equation*}
        \|\mathcal{N}_{NR}(u_1,\dots,u_p)\|_{X_T^{s+\sigma,-b'}}\leq  C\prod_{l=1}^{p}\|u_l\|_{X_T^{s,b}}.
    \end{equation*}
\end{lemma}

Note that the nonlinear term in the NLS equation is $|u|^{p-1}u=\mathcal{N}(u,\dots ,u)$. Then the above multilinear estimate, combined with Lemma~\ref{lemBourgain_S(t)estimate} and phase shift $u(t)\mapsto e^{i\theta_u(t)}u(t)$ (which serves to remove the ``worst" single-resonant part $\mathcal{N}_R(u,\dots ,u)$), leads to nonlinear smoothing: 

\begin{proposition}[{\cite[Theorem 3.1]{CXZZ-25}}]\label{prop_H1}
    Let $R,\sigma>0$ be arbitrarily given. Then there exist constants $C,\kappa>0$, and a bounded subset $Y\subset H^{1+\sigma}(\mathbb{T})$, with the following property. Assume the external force $f\colon \mathbb{R}^+\to H^{1+\sigma}(\mathbb{T})$ satisfies
    \[\int_{n-1}^n \|f(t)\|_{H^{1+\sigma}(\mathbb{T})}^2\, dt\le R\quad \text{for any }n\in \mathbb{N}.\]
    Then for any $u_0\in H^1$, the solution $u(t)$ of NLS equation \eqref{deterministicNLS} satisfies
    \[\dist_{H^1} (u(t),Y)\le C(1+E(u_0))e^{-\kappa t}\quad \text{for any } t\geq 0.\]
\end{proposition}

We clarify that any bounded attractor $Y\subset H^{1+\sigma}$ with $\sigma>0$ arbitrarily small suffices for the verification of hypothesis $(\mathbf{H1})$ as well as $(\mathbf{H2})$ below. Nevertheless, in order to rigorously derive $(\mathbf{H4})$, we need extra regularity $H^3$ for $Y$ (namely $\sigma=2$) to justify some computations.

\subsubsection{Asymptotic compactness of linearization}\label{subsubsec_H2}

We turn to the proof of Proposition~\ref{prop_H2}. Inspired by \cite{CXZZ-25}, we apply resonance decomposition to the potential terms in \eqref{D_uS}, namely
\[\tfrac{p+1}{2}|u|^{p-1}v+\tfrac{p-1}{2}|u|^{p-3} u^2\bar v=\tfrac{p+1}{2}\mathcal{N}(u,\dots ,u,v)+\tfrac{p-1}{2}\mathcal{N}(u,\dots,u,v,u).\]
Note that the corresponding single resonance parts are
\[\tfrac{p+1}{2}\mathcal{N}_R(u,\dots,u,v)=\tfrac{p+1}{4\pi} \|u\|_{L^{p-1}(\mathbb{T})}^{p-1} v+\tfrac{p^2-1}{8\pi}u\int_{\mathbb{T}}|u|^{p-3}\bar{u}v,\]
\[\tfrac{p-1}{2}\mathcal{N}_R(u,\dots ,u,v,u)=\tfrac{p^2-1}{8\pi}u\int_{\mathbb{T}}|u|^{p-3}u\bar{v}.\]
If we denote with
\[\NR(u,v)=\tfrac{p+1}{2}\mathcal{N}_{NR}(u,\dots ,u,v)+\tfrac{p-1}{2}\mathcal{N}_{NR}(u,\dots ,u,v,u).\]
Then the potential terms can be written as
\[\tfrac{p+1}{2}|u|^{p-1}v+\tfrac{p-1}{2}|u|^{p-3} u^2\bar v=\tfrac{p+1}{4\pi} \|u\|_{L^{p-1}(\mathbb{T})}^{p-1} v+\tfrac{p^2-1}{4\pi} u\re\, (|u|^{p-3}u,v)_{L^2(\mathbb{T})}+\NR(u,v).\]
% For the third term, for $b>1/2$ sufficiently close to $1/2$, the multilinear estimate \eqref{lem_nonlinearsmoothing} yields
% \begin{equation}\label{NRestimate}
%     \|\NR(u,v)\|_{X_1^{s+\sigma,b-1}}\le C\|u\|^{p-1}_{X_1^{s,b}} \|v\|_{X_1^{s,b}}.
% \end{equation}
% The second terms also belongs to $H^{5/4}$ owing the extra regularity that $u(t)\in X_1^{s+\sigma,b}\subset C(0,1;H^{5/4}(\mathbb{T}))$ (note that $\re\, (|u|^{p-3}u,v)_{L^2(\mathbb{T})}$ is merely a function of $t$ and does not depend on $x$). In addition, the first term can be removed by phase shift as in last subsection.

\begin{proof}[\rm \bf Proof of Proposition~\ref{prop_H2}]
    Let us define the auxiliary functions
\[V(t,x):=e^{i\theta_u (t)} v(t,x)\quad \text{and}\quad U(t,x):=e^{i\theta_u(t)} u(t,x).\]
Then $V(0,x)=v_0(x)$, and $V(t)$ satisfies the equation
\[iV_t+V_{xx}+ia(x)V=\tfrac{p^2-1}{4\pi} U \re\, (|U|^{p-3}U,V)_{L^2(\mathbb{T})}+\NR(U,V).\]
Without loss of generality, we may assume that the parameter $b>1/2$ is sufficiently close to $1/2$, so that Lemma~\ref{lem_nonlinearsmoothing} is valid for $b'=b-1$, which yields
    \begin{equation}\label{NRestimate}
    \|\NR(U,V)\|_{X_1^{5/4,b-1}}\le C\|U\|^{p-1}_{X_1^{1,b}} \|V\|_{X_1^{1,b}}.
\end{equation}
In view of Lemma~\ref{Lemma-multiplication}, we have
\begin{equation}\label{Uestimate}
    \|U\|_{X_1^{5/4,b}}\le C\|e^{i\theta_u(t)}\|_{H^1(0,1)}\|u\|_{X_1^{5/4,b}}\le C,\end{equation}
where we use the assumption $\|u\|_{X_1^{5/4,b}}\le R$ and that
\[\|e^{i\theta_u(t)}\|_{H^1(0,1)}^2= 1+\|\theta_u(t)\|_{L^2(0,1)}^2\le 1+C\|u\|_{L^\infty(0,1;L^\infty(\mathbb{T}))}^{2(p-1)}\le 1+C\|u\|_{X_1^{5/4,b}}^{2(p-1)}\le C.\]
And similarly, taking Lemma~\ref{lem_wplinearization}(a) into account, we have
\begin{equation}\label{Vestimate}
    \|V\|_{X_1^{1,b}}\le C\|e^{i\theta_u(t)}\|_{H^1(0,1)}\|v\|_{X_1^{1,b}}\le C\|v_0\|_{H^1}.
\end{equation}

By virtue of Duhamel formula and Lemma~\ref{lemBourgain_S(t)estimate}, we find
\begin{equation}\label{duhamelestimate-1}
    \|V(1)-S_a(1) v_0\|_{H^{5/4}}\le C\left[ \|U\re\, (|U|^{p-3}U,V)_{L^2(\mathbb{T})}\|_{X_1^{5/4,b-1}}+\|\NR(U,V)\|_{X_1^{5/4,b-1}}\right].
\end{equation}
To treat the first term, notice that $L^\infty(0,1;H^{5/4})\hookrightarrow X_1^{5/4,b-1}$, and $\re\, (|U|^{p-3}U,V)_{L^2(\mathbb{T})}$ only depends on $t$. Hence (note that this is the place we essentially need extra regularity $u\in X^{5/4,b}_1$)
\begin{equation}\label{duhamelestimate-2}
\begin{aligned}
    \|U\re\, (|U|^{p-3}U,V)_{L^2(\mathbb{T})}\|_{X_1^{5/4,b-1}}&\le C\|U\|_{L^\infty(0,1;H^{5/4})} \|\re\, (|U|^{p-3}U,V)_{L^2(\mathbb{T})}\|_{L^\infty (0,1)}\\
    &\le C\|U\|_{X_1^{5/4,b}} \|U\|_{L^\infty(0,1;L^\infty(\mathbb{T}))}^{p-2} \|V\|_{L^\infty(0,1;L^\infty(\mathbb{T}))}\\
    &\le C\|U\|_{X_1^{5/4,b}}^{p-1} \|V\|_{X_1^{1,b}}\le C\|v_0\|_{H^1}.
\end{aligned}
\end{equation}
Meanwhile, the other term can be estimated by \eqref{NRestimate}--\eqref{Vestimate}:
\begin{equation}\label{duhamelestimate-3}
    \|\NR(U,V)\|_{X_1^{5/4,b-1}}
    \le C\|U\|^{p-1}_{X_1^{1,b}} \|V\|_{X_1^{1,b}}\le C\|v_0\|_{H^1}.
\end{equation}
Now the conclusion follows by substituting \eqref{duhamelestimate-2} and \eqref{duhamelestimate-3} into \eqref{duhamelestimate-1}. We emphasize that the constant $C$ is determined by $R$, rather than the profile of $u$.
\end{proof}

\section{Geometric control approach}\label{sec_geometriccontrol}

The hypothesis $(\mathbf{H4})$ is associated with an approximate control problem. Let $B\subset \mathbb{Z}$ be a finite subset, and $\mathcal{H}:={\rm span}_\mathbb{C}\, \{e_k:k\in B\}$, which is a subspace of $H^1(\mathbb{T})$. The main result of this section is stated as follows. Recall that $S\colon H^1\times L^2(0,1;H^1)\to H^1$ stands for the time-$1$ solution map of the NLS \eqref{deterministicNLS}, and $D_fS$ is equal to the solution map of linearized system \eqref{D_etaS}.

\begin{proposition}\label{prop_H4}
    Assume $u_0\in H^3(\mathbb{T})$, the external force $f\colon [0,1]\to \mathcal{H}$ is Lipschitz-observable (see Definition~{\rm\ref{def_obs}}), and the finite set $B\subset \mathbb{Z}$ is saturating (see Definition~{\rm\ref{def_saturating}}). Then 
    \[D_f S(u_0,f)|_{L^2(0,1;\mathcal{H})}\colon L^2(0,1;\mathcal{H})\to H^1(\mathbb{T})\]
    has dense image. More precisely, for any $v_1\in H^1(\mathbb{T})$ and $\varepsilon>0$, there exists a family
    \[g_k\in L^2(0,1;\mathbb{C}),\quad k\in B,\]
    so that the solution of controlled linearized system
    \begin{equation}\label{controlsystem}
        \left\{\begin{array}{ll}
    iv_t+v_{xx}+ia(x) v=\frac{p+1}{2} |u|^{p-1} v+\frac{p-1}{2} |u|^{p-3} u^2 \bar{v}+\sum_{k\in B} g_k(t) e_k(x),\\
    v(0,\cdot)=0,
    \end{array}\right.
    \end{equation}
    satisfies $\|v(1)-v_1\|_{H^1}\le \varepsilon$, where $u$ stands for the solution of NLS equation \eqref{deterministicNLS} with force $f$.
\end{proposition}

To this end, we first review the notion of Lipschitz-observability from \cite{KNS-20} in Section~\ref{subsec_Lipobservable}, and then employ the geometric control approach \cite{AS-06,Sar-12} in Section~\ref{subsec_degeneratecontrol}. Note that the extra regularity of reference solution $u_0\in H^3$ is crucial to the proof of this $H^1$ controllability result.

\subsection{Lipschitz observability}\label{subsec_Lipobservable}

Recall the following definition from \cite{KNS-20}.

\begin{definition}[Lipschitz-observability]\label{def_obs}
    Let $B\subset \mathbb{Z}$ be a finite set, and $\mathcal{H}={\rm span}_{\mathbb{C}} \{e_k:k\in B\}$. A function $f\in L^2(0,1;\mathcal{H})$ of the form\footnote{Indeed, $f_k$ is the $k$-th Fourier coefficient of $f$. For the sake of simplicity, we write $f_k(t)$ instead of $\widehat{f(t)}(k)$.}
    \[f(t,x)=\sum_{k\in B} f_k(t) e_k(x)\]
    is called Lipschitz-observable, if for any Lipschitz functions $a_k,b_k\colon [0,1]\to \mathbb{C}$ and continuous function $c\colon [0,1]\to \mathbb{C}$,
    the equality
    \begin{equation}\label{observable-1}
        \sum_{k\in B} a_k(t) \re f_k(t)+b_k \im f_k(t)=c(t)\quad \text{in }L^1(0,1)
    \end{equation}
    implies that $a_k(t)\equiv b_k(t)\equiv c(t)\equiv 0$ for any $k\in B$.
\end{definition}

\begin{remark}
    When equipping $\mathcal{H}$ with the real $L^2$-inner product $(f,g):=\re \, (f,g)_{L^2(\mathbb{T})}$, and considering the orthonormal basis $\{e_k,ie_k\}\, (k\in B)$ over $\mathbb{R}$, the corresponding coefficients in the expansion of $f$ are $\re f_k$ and $\im f_k$, respectively. Thus our formulation of Lipschitz-observability is compatible with {\rm\cite[Definition~4.1]{KNS-20}}.
\end{remark}

Heuristically, functions with suitable ``roughness" are promising to be observable. For example, the denseness of discontinuous points may lead to observability.

\begin{example}\label{eg_densediscontinuous}
    Denote the set of discontinuous points of $\re f_k$ and $\im f_k$ by $P_k$ and $Q_k$, respectively. Assume these $P_k$ and $Q_k\, (k\in B)$ are all dense in $[0,1]$, and mutually disjoint. Then for \eqref{observable-1} to hold, $a_k(t)\re f_k(t)$ must be continuous at any $t\in P_k$, which forces $a_k(t)=0$ on this dense set. As $a_k$ is continuous, we find that $a_k\equiv 0$. Similarly, we have $b_k\equiv c\equiv 0.$
\end{example}

\medskip

For less trivial examples, let us recall the notion of observable measure from \cite{KNS-20}.

\begin{definition}[Observable measure]\label{def_obsmeas}
    Let $B\subset \mathbb{Z}$ be a finite set, and $\mathcal{H}={\rm span}_{\mathbb{C}} \{e_k:k\in B\}$. A probability measure $\ell \in \mathcal{P}(L^2(0,1;\mathcal{H}))$ is called Lipschitz-observable, if $\ell$-almost every trajectory $\eta\in L^2(0,1;\mathcal{H})$ is Lipschitz-observable. In addition, we say an $L^2(0,1;\mathcal{H})$-valued random variable is Lipschitz-observable, if and only if its law is a Lipschitz-observable measure.
\end{definition}

We provide two examples on the observability of complex Haar processes.

\begin{definition}[Haar process]
    The $L^\infty$-normalized Haar system $\{h_0,h_{jl}\}\, (j\in \mathbb{N},\ l\in \mathbb{N}_0)$ is defined by \begin{equation}\label{def_Haar}
    h_0(t)=\begin{cases}
        1,&0\le t<1,\\
        0,&\text{elsewhere.}
    \end{cases},\quad 
    h_{jl}(t)=\begin{cases}
        1,&l2^{-j}\le t<(l+1/2)2^{-j},\\
        -1,&(l+1/2)2^{-j}\le t<(l+1)2^{-j},\\
        0,&\text{elsewhere.}
    \end{cases}
\end{equation}
Let $B\subset \mathbb{Z}$ be a finite set, and $\mathcal{H}={\rm span}_{\mathbb{C}} \{e_k:k\in B\}$. A complex Haar process is an $L^2(0,1;\mathcal{H})$-valued random variable, which can be formulated as
    \begin{equation}\label{Haardef}
        \eta(t)=\sum_{k\in B} b_k \left((\xi_{k0}^1+i\xi_{k0}^2) h_0(t)+\sum_{j=1}^\infty  \sum_{l=0}^{2^j-1} c_j (\xi_{kjl}^1+i\xi_{kjl}^2) h_{jl}(t)\right)e_k(x).
    \end{equation}
    Here $b_k$ and $c_j$ are non-zero constants with $\sum_{j=1}^\infty |c_j|^2<\infty$, and $\xi_{k0}^1, \xi_{k0}^2, \xi_{kjl}^1, \xi_{kjl}^2$ are real-valued i.i.d.~random variables with continuous density supported in $[-1,1]$.
\end{definition}

Note that $\{h_0,h_{jl}\}\, (j\in \mathbb{N},\ 0\le l\le 2^j-1)$ is an orthonormal basis of $L^2(0,1;\mathbb{R})$. The reader is referred to \cite[Section 5.2]{KNS-20} for the proof of following two examples. Clearly the latter is associated with the Main Theorem (see the noise structure $(\mathbf{S2})$).

\begin{example}[Haar process with exponential decay coefficients]
    If $c_j=A^{-j}$ with $A>1$ and sufficiently close to $1$, then the Haar process is Lipschitz-observable.
\end{example}

\begin{example}[Haar process with polynomial decay coefficients]\label{eg_haar}
    If $c_j=cj^{-q}$ with $c>0$ and $q>1$ arbitrarily given, then the Haar process is Lipschitz-observable.
\end{example}

\subsection{Approximate controllability}\label{subsec_degeneratecontrol}

We present the proof of Proposition~\ref{prop_H4} in this subsection. To this end, let us first give the definition of saturating subspaces.

\begin{definition}[Saturating sets]\label{def_saturating}
    Given a subset $B\subset \mathbb{Z}$, set $B_0=B$ and define iteratively that
    \[B_n=B_{n-1}\cup \{2k-l:k\in B_0,\ l\in B_{n-1}\}.\]
    Then $B$ is called saturating if and only if $\bigcup_{n=0}^\infty B_n=\mathbb{Z}$.
\end{definition}

Note that we only allow $k\in B$, while $l$ may range over the former extensions $B_{n-1}$. We provide an easy but important example: a saturating set containing two elements.

\begin{example}[Two elements of saturation]
    For any $n\in \mathbb{Z}$, the set $B:=\{n,n+1\}$ is saturating. Indeed, one can prove by induction that
    \[B_k\supset [n-k,n+k+1]\cap \mathbb{Z}.\]
    This is due to the observation that
    \[2 n-(n+k+1)=n-k-1\quad \text{and}\quad 2(n+1)-(n-k)=n+k+2.\]
    In particular, a simplest choice is $B=\{0,1\}$, which corresponds to Fourier modes $1$ and $e^{ix}$.
\end{example}

Now we turn to the proof of Proposition~\ref{prop_H4}. Let $A:=D_f S(u,f)$. It suffices to show
\[AP_{L^2(0,1;\mathcal{H})}\colon L^2(0,1;H^1)\to H^1\]
has dense image, where $P_{L^2(0,1;\mathcal{H})}\colon L^2(0,1;H^1)\to L^2(0,1;\mathcal{H})$ denotes the orthogonal projection. Equivalently, this means its (Banach) adjoint
\[P_{L^2(0,1;\mathcal{H})} A^*\colon H^{-1}\to L^2(0,1;H^{-1})\]
has trivial kernel. Here we slightly abuse the notation by still writing $P_{L^2(0,1;\mathcal{H})}$ for the orthonormal projection from $L^2(0,1;H^{-1})$ to $L^2(0,1;\mathcal{H})$. In view of \eqref{ADuhamel}, we have
\[P_{L^2(0,1;\mathcal{H})} A^*=\int_0^1 P_{\mathcal{H}} \mathcal{R}_u(1,t)^*\, dt,\]
where $P_{\mathcal{H}}\colon H^{-1}\to \mathcal{H}$ represents the orthonormal projection.

We shall demonstrate that $\mathcal{R}_u(1,t)^*$ can be identified with the solution of backward system
\begin{equation}\label{adjointsystem}
\left\{\begin{array}{ll}
i\varphi_t+\varphi_{xx}-ia(x)\varphi=\frac{p+1}{2}|u|^{p-1}\varphi-\frac{p-1}{2}|u|^{p-3} u^2\bar \varphi,\\
\varphi(1,x)=\varphi_1(x)\in H^{-1}.
\end{array}
\right.
\end{equation}
See Lemma~\ref{lem_wplinearization}(b) for the well-posedness of this backward system. In the sequel, we denote the corresponding solution map from time $1$ to $t\in [0,1]$ by $\mathcal{R}^{bw}_u (t,1) \varphi_1(x)$.

\begin{lemma}\label{lem_dual}
    For any $v_0\in H^1$ and $\varphi_1\in H^{-1}$, set $v(t)=\mathcal{R}_u(t,0)v_0$ and $\varphi(t)=\mathcal{R}^{bw}_u (t,1) \varphi_1$. Then the $(H^1,H^{-1})$ pairing\footnote{In order to associate this pairing with $L^2$-inner product, we set $\langle f,g\rangle_{H^1,H^{-1}}=\int_{\mathbb{T}} f\bar{g}$ for smooth $f,g$ and extend over $f\in H^1$ and $g\in H^{-1}$ by continuity. This yields a real-bilinear functional after taking the real part.} between $v(t)$ and $\varphi(t)$, namely
    \[\re\, \langle v(t),\varphi(t)\rangle_{H^1,H^{-1}},\]
    is a constant for $t\in [0,1]$.
\end{lemma}

\begin{proof}[\rm \bf Proof of Lemma~\ref{lem_dual}]
    Owing to a standard approximation argument, without loss of generality, we may assume $v_0,\varphi_1\in H^1$, and hence $v,\varphi\in C(0,1;H^1)$ and $v_t,\varphi_t\in C(0,1;H^{-1})$. As a result, the pairing reduces to $L^2$-inner product, and exchanging spatial-integrations with time-derivatives is admissible. Taking the equations of $\varphi$ and $v$ into account, we find
    \begin{align*}
        \frac{d}{dt}\re\, (v(t),\varphi(t))_{L^2}&=\re\, (v_t,\varphi)_{L^2}+\re\,(v,\varphi_t)_{L^2}\\
        &=\re\, (iv_{xx}-a(x)v-i\tfrac{p+1}{2}|u|^{p-1} v-i\tfrac{p-1}{2}|u|^{p-3} u^2 \bar{v},\varphi)_{L^2}\\
        &\quad +\re\, (v,i\varphi_{xx}+a(x)\varphi-i\tfrac{p+1}{2}|u|^{p-1} \varphi+i\tfrac{p-1}{2} |u|^{p-3} u^2 \bar{\varphi})_{L^2}
    \end{align*}
    Note that
    \[(iv_{xx},\varphi)_{L^2}=i\int_{\mathbb{T}} v_{xx}\bar{\varphi}=i\int_{\mathbb{T}} v\overline{\varphi_{xx}}=-(v,i\varphi_{xx})_{L^2},\]
    \[(-a(x)v,\varphi)_{L^2}=-\int_{\mathbb{T}} a(x)v\bar{\varphi}=-(v,a(x)\varphi)_{L^2},\]
    \[(-i|u|^{p-1} v,\varphi)_{L^2}=-i\int_{\mathbb{T}} |u|^{p-1} v\bar{\varphi}=(v,i|u|^{p-1} \varphi)_{L^2}.\]
    Thus we find
    \begin{align*}
        \frac{d}{dt}\re\, (v(t),\varphi(t))_{L^2}&=\re\, (-i\tfrac{p-1}{2}|u|^{p-3} u^2 \bar{v},\varphi)_{L^2}+\re\, (v,i\tfrac{p-1}{2} |u|^{p-3} u^2 \bar{\varphi})_{L^2}\\
        &=\frac{p-1}{2}\int_{\mathbb{T}} |u|^{p-3}[\re\, (-i u^2\bar{v}\bar{\varphi})+\re (-i \bar{u}^2 v\varphi)]=0.
    \end{align*}
    Here in the last step we use that the complex conjugate of $-i u^2 \bar{v}\bar{\varphi}$ is $i\bar{u}^2 v \varphi$.
\end{proof}

We can interpret this lemma as for any $v_0\in H^1$ and $\varphi_1\in H^{-1}$,
\begin{equation}\label{dualidentity}
    \re\, \langle \mathcal{R}_u(1,t)v_0,\varphi_1(x)\rangle_{H^1,H^{-1}}=\re\, \langle v_0,\mathcal{R}_u^{bw}(1,t)\varphi_1\rangle_{H^1,H^{-1}}.
\end{equation}
Recall our convention on scaler fields: $H^s$ is equipped with the real inner-product $\re\, (\cdot,\cdot)_{H^s}$. Thus $H^{-1}$ can be identified with the Banach dual of $H^1$ via the map
\[H^{-1}\ni \varphi\mapsto \re\, \langle \cdot,\varphi \rangle_{H^1,H^{-1}}.\]
Therefore the identity \eqref{dualidentity} implies
\[\mathcal{R}_u(1,t)^*=\mathcal{R}^{bw}_u(t,1).\]
And as a corollary, we obtain the following:

\begin{corollary}
    The adjoint of operator $AP_{L^2(0,1;\mathcal{H})}$ has the form of
    \[P_{L^2(0,1;\mathcal{H})}A^*=\int_0^1 P_{\mathcal{H}} \mathcal{R}_u^{bw}(1,t)\, dt.\]
\end{corollary}

We are now in a position to accomplish Proposition~\ref{prop_H4}. To show that $AP_{L^2(0,1;\mathcal{H})}$ has dense image, a more convenient description is that the self-adjoint operator
\[G:=AP_{\mathcal{H}}(AP_{\mathcal{H}})^*=AP_{\mathcal{H}} A^*=\int_0^1 R_u(1,t)P_{\mathcal{H}} R_u^{bw}(1,t)\, dt\]
has trivial kernel. In control theory, $G$ is called the Gramian matrix (see, e.g., \cite{Coron-07}), while in the terminology of Malliavin calculus it is also known as the Malliavin matrix (see, e.g., \cite{HM-06}).

\begin{proof}[\rm \bf Proof of Proposition~\ref{prop_H4}]
As illustrated above, it suffices to show that $\varphi_1=0$ whenever $G(\varphi_1)=0$. Indeed, for any $\varphi_1\in \ker G$, we have
\[0=\re\,\langle G(\varphi_1),\varphi_1\rangle_{H^1,H^{-1}}=\int_0^1 \|P_{\mathcal{H}} \mathcal{R}_u^{bw}(1,t)\varphi_1\|_{L^2}^2\, dt\]
Thus $P_{\mathcal{H}} \mathcal{R}^{bw}_u(t,1) \varphi_1\equiv 0$. Alternatively, set $\varphi(t)=\mathcal{R}_u^{bw}(1,t) \varphi_1$, then for each $l\in B$, we have\footnote{We slightly abuse the notation that, by $(\varphi(t),e_l)_{L^2(\mathbb{T})}$ we actually mean the pairing between $H^{-1}$ and $H^1$. These two notions coincides when $\varphi\in L^2(\mathbb{T})$. As $e_l$ is smooth, when taking time derivative of \eqref{orthB_0} later, the expression $(\varphi_t,e_l)_{L^2(\mathbb{T})}$ is well-defined and belongs to $C(0,1;\mathbb{C})$ since $\varphi_t\in C(0,1;H^{-1})$.}
\begin{equation}\label{orthB_0}
    (\varphi(t),e_l)_{L^2}=0\quad \text{for any } t\in [0,1].
\end{equation}

\noindent{\bf Claim.} If \eqref{orthB_0} is valid with some $l\in \mathbb{Z}$, then for any $k\in B$, it also holds with $l$ replaced by $2k-l$.

Once the claim has been demonstrated, we can prove by induction that \eqref{orthB_0} holds for $l\in B_n$ for any $n\in \mathbb{N}_0$ (see Definition~\ref{def_saturating}). Thanks to our assumption that $B$ is saturating, we conclude with $\varphi(t)\equiv 0$, and in particular $\varphi_1=\varphi(1)=0$ as desired. It thus remains to fulfill our claim. 

\medskip

\noindent\textbf{Proof of the claim.} Recall we denote the $p$-multiplication operator by \eqref{p_product}. And when $u_1=\cdots =u_p=u(t)$, we simply write $\mathcal{N}(u(t))=|u(t)|^{p-1}u(t)$. The $k$-th differential map of $\mathcal{N}(u(t))$ is $D^k\mathcal{N}_{u(t)}\colon \mathbb{C}^k\to \mathbb{C}$, which is a symmetric real-linear form. It is readily seen that $D^k \mathcal{N}_{u(t)}(h_1,\dots ,h_k)$ is a $p$-homogeneous polynomial of $u$ and $h_1,\dots,h_k$ (as well as their complex conjugates). For example, the first and the $p$-th order differential maps are
    \[D\mathcal{N}_{u(t)} (h_1)=\tfrac{p+1}{2}|u|^{p-1} h_1+\tfrac{p-1}{2}|u|^{p-3} u^2 \bar{h}_1,\]
    \[D^p\mathcal{N}_{u(t)} (h_1,\dots ,h_p)=\sum_{\rm permutations} \mathcal{N}(h_{j_1},\dots ,h_{j_p}).\]
    Here the last summation is taken over all $p!$ permutations $(j_1,\dots ,j_p)$ of $(1,\dots ,p)$.
    Specifically, we find that the backward system \eqref{adjointsystem} of $\varphi$ can be rewritten as
    \begin{equation}\label{adjointsystem-variant}
        i\varphi_t+\varphi_{xx}-ia(x)\varphi=-iD\mathcal{N}_{u(t)} (i\varphi)
    \end{equation}
    
    Computing the time derivative of \eqref{orthB_0}, and taking \eqref{adjointsystem-variant} into account, we find
    \[(\varphi_{xx}-ia(x)\varphi+iD\mathcal{N}_{u(t)}(i\varphi),e_l)_{L^2}=0.\]
    And integration by parts yields
    \begin{equation}\label{AS-1}
        \big(\varphi,(-|l|^2+ia(x)) e_l\big)_{L^2}+i(D\mathcal{N}_{u(t)}(i\varphi),e_l)_{L^2}=0.
    \end{equation}
    We point out that the NLS equation \eqref{deterministicNLS} of $u$ leads to
    \[u(t)=u_0+\int_0^t (iu_{xx}-a(x) u-i|u|^2 u)\, ds-i\int_0^t f(s)\, ds.\]
    This can be written in the form
    \[u(t)=:y(t)-i\int_0^t f(s)\, ds=y(t)-i\sum_{k\in B}\int_0^t f_k(t)e_k(x)\, dt,\]
    where $y\in C^1(0,1;H^1)$ thanks to the assumption $u_0\in H^3$ so that $u\in C(0,1;H^3)$. Therefore, taking time derivative of \eqref{AS-1}, and applying the chain rule to the term $D\mathcal{N}_{u(t)}(i\varphi)$, we get
    \[\big(\varphi_t,(-|l|^2+ia(x))e_l\big)_{L^2}+i(D\mathcal{N}_{u(t)} (i\varphi_t),e_l)_{L^2}+i(D^2\mathcal{N}_{u(t)}(u_t,i\varphi),e_l)_{L^2}=0.\]
    Split $u_t$ by $y_t-i\sum_{k\in B} f_k(t) e_k(x)$ in the last term, we thus obtain 
    \[\sum_{k\in B} a_k(t) \re f_k(t)+b_k(t) \im f_k(t)=c(t),\]
    where
    \[a_k(t)=i(D^2\mathcal{N}_{u(t)}(ie_k,i\varphi),e_l)_{L^2},\]
    \[b_k(t)=-i(D^2\mathcal{N}_{u(t)}(e_k,i\varphi),e_l)_{L^2},\]
    \[c(t)=\big(\varphi_t,(-|l|^2+ia(x))e_l\big)_{L^2}+i(D\mathcal{N}_{u(t)} (i\varphi_t),e_l)_{L^2}+i(D^2 \mathcal{N}_{u(t)}(y_t,i\varphi),e_l)_{L^2}.\]
    Note that $D^k\mathcal{N}_{u(t)}$ is homogeneous polynomial of degree $p-k$ in $u$ and degree $1$ in each of the $k$ arguments. It is thus easy to see that $a_k,b_k$ belongs to $C^1(0,1;\mathbb{C})$ and $c$ is continuous. 
    Now the Lipschitz-observability of $f$ implies $a_k(t)\equiv b_k(t)\equiv c(t)\equiv 0$ for any $k\in B$. 
    
    To iterate this argument, we take the derivative of $a_k(t)\equiv 0$. A similar reasoning leads to
    \[(D^3 \mathcal{N}_{u(t)} (ie_{k_1},ie_{k_2},i\varphi),e_l)_{L^2(\mathbb{T})}\equiv 0\quad \text{for any }k_1, k_2\in B.\]
    In the same manner, for any $k_1,\dots ,k_{p-1}\in B$, we find that
    \[(D^p \mathcal{N}_{u(t)}(ie_{k_1},\dots ,ie_{k_{p-1}},i\varphi),e_l)=0.\]
    If we pick $k_1=\cdots =k_{p-1}=k\in B$, then
    \[D^p\mathcal{N}_{u(t)}(ie_k,\dots ,ie_k,i\varphi)=i\tfrac{(p-1)!}{(2\pi)^{(p-1)/2}} \big[\tfrac{p+1}{2}\varphi+\tfrac{p-1}{2}e^{2ikx} \bar{\varphi}\big].\]
    Taking \eqref{orthB_0} into account, we obtain that $(e_l,e_k^2 \bar{\varphi})_{L^2(\mathbb{T})}=0$, or equivalently,
    \[(\varphi(t), e_k^2 \bar{e}_l)_{L^2(\mathbb{T})}=\tfrac{1}{2\pi}(\varphi(t), e_{2k-l})_{L^2(\mathbb{T})}=0.\]
    The arbitrariness of $k\in B$ concludes our claim.
\end{proof}

We mention that the proof above is similar in spirit to \cite[Section 4.2]{KNS-20}, which dealt with the quintic complex Ginzburg–Landau equation, but the mechanism of saturation is different.

\section{Exponential mixing for NLS equations}\label{sec_applyNLS}

We apply the AET-based criterion for exponential mixing (Theorem~\ref{thm_criterion}) to random NLS equation \eqref{randomNLS}. To this end, we show that the Markov chain generated by the solutions $u(n)$ at integral times $n$, fits into the setting of random dynamical system in Section~\ref{subsec_RDSsetting}, and then verify the hypotheses $(\mathbf{H1})$--$(\mathbf{H5})$, for which the prerequisites has been addressed in last two sections.

The content of Main Theorem can be easily extended to a wider range of noises. The core assumptions are decomposability and observability.

\begin{itemize}
\item[$(\mathbf{S2})'$] (Noise structure: decomposable and observable) Let $B\subset \mathbb{Z}$ be a saturating subset (see Definition~\ref{def_saturating}), and equip $\mathcal{H}:={\rm span}_{\mathbb{C}} \{e_k:k\in B\}$ with the $H^3$-norm.

The random noise $\eta(t,x)$ has the form
\[\eta(t,x)=\sum_{n=1}^\infty \mathbf{1}_{[n-1,n)}(t) \eta_n(t-n+1,x),\]
where $\eta_n$ is a sequence of $L^2(0,1;\mathcal{H})$-valued i.i.d.~random variables. In addition,

\begin{itemize}
    \item The common law of $\eta_n$, denoted $\ell\in \mathcal{P}(L^2(0,1;\mathcal{H}))$, is observable (see Definition~\ref{def_obsmeas}).
    
    \item Decomposability hypothesis $(\mathbf{H5})$ holds for $\eta_n$.

    \item The constant function $0$ belongs the support of $\ell$.
\end{itemize}
\end{itemize}

Now we state a generalization of the Main Theorem.

\begin{theorem}\label{thm_NLSgeneral}
    Assume the damping $a(x)$ and noise $\eta(t,x)$ satisfy the settings $(\mathbf{S1})$ and $(\mathbf{S2})'$, respectively. Denote by $u_n=u(n)$, where $u(t)$ is the solution of random NLS \eqref{randomNLS}. Then the Markov process $(u_n,\mathbb{P}_u)$ on $H^1(\mathbb{T})$ admits a unique invariant measure $\mu\in \mathcal{P}(H^1(\mathbb{T}))$ with compact support, and exponential mixing holds in the sense of \eqref{EM-NLS}.
\end{theorem}

\begin{proof}[\rm\bf Proof of the Main Theorem]
    In view of Example~\ref{eg_haar}, Haar processes with polynomial decay coefficients is Lipschitz-observable. So $(\mathbf{S2})$ is a special case of $(\mathbf{S2})'$. In particular, except for the assertion that $\mu(H^1\setminus C^\infty)=0$, the Main Theorem is a direct consequence of Theorem~\ref{thm_NLSgeneral}. 
    
    This final assertion follows from Proposition~\ref{prop_H1}. Indeed, for any $\sigma>0$, in view of this proposition, since $\supp(\mu)$ is bounded, if $u_0\in \supp(\mu)$, then almost surely
    \[\dist_{H^1}(u_n,Y_\sigma)\le C(1+E(u_0))e^{-\kappa n}\le Ce^{-\kappa n},\]
    where $Y_\sigma$ is a bounded subset of $H^{1+\sigma}$, and $C$ is independent of $u_0\in \supp(\mu)$. Thus
    \[P_n(u,B_{H^1}(Y_\sigma,Ce^{-\kappa n}))=1\quad \text{for any }u\in \supp(\mu).\]
    Therefore, in view of $\mu=P_n^* \mu$, we find
    \[\mu (B_{H^1}(Y_\sigma,Ce^{-\kappa n}))=\int_{\supp(\mu)} P_n (u,B_{H^1}(Y_\sigma,Ce^{-\kappa n}))\, \mu(du)=1.\]
    Since $Y_\sigma$ is closed in $H^1$, sending $n$ to $\infty$ yields
    \[\mu(Y_\sigma)=1.\]
    Now we conclude that $\mu(H^1\setminus C^\infty)=0$, thanks to $Y_\sigma\subset H^{1+\sigma}$ and $\bigcap_{\sigma>0} H^{1+\sigma}=C^\infty$.
\end{proof}

The remainder of this section is devoted to the proof Theorem~\ref{thm_NLSgeneral}. Since the deterministic NLS equation is globally well-posed, it immediately follows that this model fits into the abstract random dynamical system setting in Section~\ref{subsec_RDSsetting}. Indeed, let
\[X=\tilde{H}^1,\quad E=L^2(0,1;\mathcal{H})\]
and $S\colon X\times E\to X$ be the time-$1$ solution map. Here $\tilde{H}^1$ stands for the Sobolev space $H^1$ equipped with the equivalent norm specified by Lemma~\ref{lem_equivalentnorm}. In view of Lemma~\ref{lem_Ssmooth}, the evolution map $S$ is smooth. In addition, we set
\[V=H^{5/4},\]
which is compactly embedded in $X$.

\begin{proof}[\rm \bf Proof of Theorem~\ref{thm_NLSgeneral}]
    Since $(\mathbf{H5})$ is guaranteed by the noise structure $(\mathbf{S2})'$, it remains to verify $(\mathbf{H1})$--$(\mathbf{H4})$, which are consequences of our study on the deterministic NLS equation.

    \medskip
    \noindent\textbf{Verification of exponential asymptotic compactness $(\mathbf{H1})$.} Thanks to $(\mathbf{H5})$ with $\mathcal{H}$ equipped with $H^3$-norm, almost surely $\|\eta_n\|_{L^2(0,1;H^3)}\le \sum_{k\in \mathbb{N}} b_k^2<\infty$. By virtue of Proposition~\ref{prop_H1}, there exists a bounded subset $Y_0\subset H^3$, and constants $C,\kappa>0$, such that
    \[\dist_{H^1}(u_n,Y_0)\le C(1+E(u_0))e^{-\kappa n}.\]
    As $H^3$ is compactly embedded in $H^1$, this is almost hypothesis $(\mathbf{H1})$, except that $Y_0$ need not be invariant. To this end, we define iteratively $Y_n=S(Y_{n-1},K)$, and set $Y$ to be the closure of $\bigcup_{n=0}^\infty Y_n$. It is not hard to see that $Y$ is invariant and compact; see, e.g., \cite[Proposition 2.2]{LWXZZ-24}.

    \medskip
    \noindent\textbf{Verification of asymptotic compactness of linearization $(\mathbf{H2})$.} Let us set
    \[T(u_0,f)(v_0):=v(1)-e^{-i\theta_u(1)}S_a(1) v_0.\]
    In view of Proposition~\ref{prop_H2}, we see that $T(u_0,f)$ is a bounded linear operator from $X=\tilde{H}^1$ to $V=H^{5/4}$. Now \eqref{ACL-1} follows from \eqref{H2-1} since $Y$ is bounded in $H^3\subset H^{5/4}$. As for \eqref{ACL-2}, by Lemma~\ref{lem_equivalentnorm}, our choice of equivalent norm yields a constant $q_0\in (0,1)$ so that
    \[\|D_{u_0}S(u_0,f)-T(u_0,f)(v_0)\|_{\tilde{H}^1}=\|v(1)-T(u_0,f)(v_0)\|_{\tilde{H}^1}=\|e^{-i\theta_u(1)}S_a(1)v_0\|_{\tilde{H}^1}\le q_0\|v_0\|_{\tilde{H}^1}.\]
    Finally, \eqref{ACL-3} is an easy consequence of local Lipschitz-continuity of solution map.

    \medskip
    
    \noindent\textbf{Verification of dissipativity $(\mathbf{H3})$.} According to Proposition~\ref{prop_stability}, as $Y$ is bounded in $X$, there exist constant $C,\beta>0$, such that for any $u_0\in Y$ and $n\in \mathbb{N}$, we have
    \[\|S_n(u_0;0,\dots ,0)\|_{H^1}\le Ce^{-\beta n} \|u_0\|_{H^1}.\]
    Choose $m_0\in \mathbb{N}$ so that $Ce^{-\beta m_0}\le q_0$, then \eqref{Dis-1} holds with $\tilde{y}=0$ and $\tilde{\zeta}_1=\cdots =\tilde{\zeta}_n=0$.

    \medskip

    \noindent\textbf{Verification of asymptotic controllability around trajectory $(\mathbf{H4})$.} This is a direct consequence of Proposition~\ref{prop_H4}, since we choose $Y$ to be a subset of $H^3$, and $\eta$ is almost surely Lipschitz-observable by our assumption $(\mathbf{S2})'$.

    \medskip

    Now we have justified all assumptions of Theorem~\ref{thm_criterion}, which directly implies Theorem~\ref{thm_NLSgeneral}.
\end{proof}

\begin{appendices}

\section*{Appendix}
% \settocdepth{subsection}

\stepcounter{section}
\numberwithin{equation}{subsection}
\numberwithin{theorem}{subsection}
\setcounter{equation}{0}
\setcounter{theorem}{0}
\renewcommand\thesubsection{\Alph{subsection}}

\subsection{Supplementary proofs for the abstract criterion}\label{appendix_auxproof}

We gather here some adaptions of existing arguments in \cite{KNS-20,NZZ-24} needed for our AET-based criterion.

\subsubsection{Control property via asymptotic compactness and approximate inverse}\label{appendix_meastrans}

This subsection exhibits the technical proof of Lemma~\ref{lem_approxcontrol}. We first follow \cite[Section 2.6]{KNS-20} and \cite[Section 7.2]{NZZ-24} to define the family of Borel sets $K^{y,\sigma}\subset K^y$ and derive \eqref{ApproxCtrl-1} and \eqref{ApproxCtrl-2}. The setting here is slightly different, while the arguments are verbatim. Specifically, Lemma~\ref{lem_inverse} below allows us to obtain uniform estimates which are independent of $y\in Y$ and $\zeta\in K^{y,\sigma}$. Then we apply the new idea illustrated in Lemma~\ref{lem_contoltoy}, exploiting the compact approximation $T(y,\zeta)$ to derive the local stabilization property \eqref{ApproxCtrl-3}. As the constants now only depend on $\sigma$, the proof of Lemma~\ref{lem_contoltoy} can be immediately adapted. More precisely, for $\zeta\in K^{y,\sigma}$, the constant $C_{y,\zeta}$ in \eqref{stabilization-1} can be improved to a constant that depends only on $\sigma$ and not on $y$ or $\zeta$.

Throughout the remainder of this subsection, we denote $A\colon Y\times E\to \mathcal{L}(E,X)$ as
\[A(y,\zeta)=D_\zeta S(y,\zeta)\colon E\to X.\]
Recall $(\psi_k)_{k\in \mathbb{N}}$ is an orthonormal basis of $E$ involved in the noise structure $(\mathbf{H5})$. Let $E_M$ be the ``low-frequency" subspace spanned by $\psi_1,\dots ,\psi_M$, and $P_M$ the orthogonal projection from $E$ to $E_M$. Then $Q_M:=\Id-P_M$ is the orthogonal projection to high frequencies $E_M^{\perp}$.

In view of hypothesis $(\mathbf{H4})$ and \cite[Theorem 2.8]{KNS-20}, we can construct an approximate inverse of $A(y,\zeta)$ with the following property. We mention that although \cite{KNS-20} assumes the family of operator $A(y,\zeta)$ to be analytic (instead of smooth), the proof still works out.

\begin{lemma}[{\cite[Theorem 2.8]{KNS-20}}]\label{lem_inverse}
    Under the assumptions of Theorem~{\rm\ref{thm_criterion}}, for any $\varepsilon\in (0,1)$, there exists $M_\varepsilon \in \mathbb{N}$, $\theta_\varepsilon,\gamma_\varepsilon,C_\varepsilon>0$, and a smooth function $\mathfrak{F}_\varepsilon \colon Y\times E\to [0,\infty)$, such that
    \begin{equation}\label{inverse-0}
        \ell(\{\mathfrak{F}_\varepsilon(y,\cdot)\le \theta_\varepsilon\})\ge 1-\varepsilon\quad \text{for any }y\in Y.
    \end{equation}
    And $\mathfrak{F}_\varepsilon (y,\cdot)$ is locally Lipschitz in the sense that for any $R>0$,
    \begin{equation}\label{mathfrakFLip}
        |\mathfrak{F}_\varepsilon(y,\zeta)-\mathfrak{F}_\varepsilon (y,\zeta')|\le C(\varepsilon,R) \|\zeta-\zeta'\|_E \quad \text{for any }y\in Y,\ \zeta,\zeta'\in B_E(R).
    \end{equation}
    
    Moreover, define $R_\varepsilon\colon Y\times E \to \mathcal{L}(X,E)$ by
    \begin{equation}\label{Rdef}
        R_\varepsilon=P_{M_\varepsilon} A^* (AA^*+\gamma_\varepsilon)^{-1},
    \end{equation}
    and define the compact set \begin{equation}\label{D_epsilon}
        D_\varepsilon:=\{(y,\zeta)\in Y\times K:\mathfrak{F}_\varepsilon(y,\zeta)\le 2\theta_\varepsilon\},
    \end{equation}
    then for any $(y,\zeta)\in D_\varepsilon$, the following estimates hold:
    \begin{equation}\label{inverse-1}
        \|R_\varepsilon (y,\zeta)\|_{\mathcal{L}(X,E)}\le C_\varepsilon,
    \end{equation}\begin{equation}\label{inverse-2}
        \|A(y,\zeta)R_\varepsilon (y,\zeta) v-v\|_X<\varepsilon \|v\|_V\quad \text{for any }v\in V.
    \end{equation}
\end{lemma}

\begin{remark}
    The auxiliary function $\mathfrak{F}_\varepsilon$ takes the form of
    \[\mathfrak{F}_\varepsilon(y,\zeta)=\sum_{j=1}^N \|A(y,\zeta) R_\varepsilon (y,\zeta)v_j-v_j\|^2,\]
    where $v_1,\dots ,v_N$ is an $(\varepsilon/4)$-net of $B_V(1)$. Specifically, it is easy to see that $\mathfrak{F}_\varepsilon$ is smooth, and the local Lipschitz-continuity \eqref{mathfrakFLip} holds. In addition, $\mathfrak{F}_\varepsilon$ characterizes how far is $AR_\varepsilon$ from the identity map in $V$, which is naturally related to \eqref{inverse-2}.
\end{remark}

With these preparations in hand, we are now able to accomplish the proof of Lemma~\ref{lem_approxcontrol}.

\begin{proof}[\rm \bf Proof of Lemma~\ref{lem_approxcontrol}]
    Fix the small parameter $\varepsilon\in (0,1)$ satisfying
    \begin{equation}\label{epsilonsmall}
        \varepsilon\le \sigma\quad \text{and}\quad q_0+C_0\varepsilon<q_1,
    \end{equation}
    where $C_0$ is the constant involved in $(\mathbf{H2})$. We apply Lemma~\ref{lem_inverse} with this specific $\varepsilon$.
    
    Recall the compact set $D_\varepsilon\subset Y\times K$ is defined by \eqref{D_epsilon}. Denote with $D_\varepsilon'=(\Id\times Q_{M_\varepsilon}) (D_\varepsilon)$ the projection of $D_\varepsilon$ to $Y\times E_{M_\varepsilon}^\perp$. Let $O_{M_\varepsilon}\subset E_{M_\varepsilon}$ be a sufficiently large open ball containing $P_{M_\varepsilon}(K)$. In the rest of the proof, we write $\zeta_{M_\varepsilon}$ and $\zeta'$ to distinguish elements of $E_{M_\varepsilon}$ and $E_{M_\varepsilon}^\perp$.
    
    According to \cite[Lemma 7.5]{NZZ-24}, $D_\varepsilon'$ can be decomposed into a disjoint union of finitely many Borel subsets $D_{\varepsilon,1}',\dots D_{\varepsilon,m}'$, and there are constants $\theta_{\varepsilon,l}\in (\theta_\varepsilon,2\theta_\varepsilon)$, so that for any $(y,\zeta')\in \overline{D_{\varepsilon,l}'}$, the number $0$ is a regular value\footnote{Given a smooth map $F\colon M\to N$ between smooth manifolds, the value $y\in N$ is called an regular value, if the tangent map $DF(x)\colon T_xM\to T_yN$ is surjective for any $x\in F^{-1}(y)$ (possibly empty).} for the smooth map 
    \begin{equation}\label{regularmap}
        O_{M_\varepsilon}\ni\zeta_{M_\varepsilon}\mapsto \mathfrak{F}_\varepsilon(y,\zeta_{M_\varepsilon}+\zeta')-\theta_{\varepsilon,l}.
    \end{equation}

    Let us now define for $y\in Y$, the Borel set $K^{y,\sigma}\subset K$ by
    \[K^{y,\sigma}=\bigcup_{l=1}^m \{\zeta\in K: (y,Q_{M_\varepsilon}(\zeta))\in D_{\varepsilon,l}',\ \mathfrak{F}_\varepsilon(y,\zeta)\le \theta_{\varepsilon,l}\}.\]
    And introduce $\Phi\colon Y\times X\times E\to E$ as
    \[\Phi^{y,x}(\zeta)=\begin{cases}
        -R_\varepsilon (y,\zeta)T(y,\zeta)(x-y),&\zeta\in K^{y,\sigma},\\
        0,&\text{elsewhere}.
    \end{cases}\]
    This is a Borel map with range contained in $E_{M_\varepsilon}$, since $R_\varepsilon$ with the form \eqref{Rdef} already has range in $E_{M_\varepsilon}$. Now it remains to demonstrate \eqref{ApproxCtrl-1}--\eqref{ApproxCtrl-3}, with $\varepsilon$ satisfying \eqref{epsilonsmall}. 
    
    Firstly, since we have assumed $\varepsilon\le \sigma$ in \eqref{epsilonsmall}, the lower bound of $\ell(K^{y,\sigma})$ \eqref{ApproxCtrl-1} is immediate. In fact, owing to $D_\varepsilon'=\bigcup_{l=1}^m D_{\varepsilon,l}'$ and $\theta_{\varepsilon,l}\ge \theta_{\varepsilon}$, it is easy to see
    \[\{\zeta\in K:\mathfrak{F}_\varepsilon(y,\zeta)\le \theta_\varepsilon\}\subset K^{y,\sigma}\quad \text{for any }y\in Y.\]
    Thanks to \eqref{inverse-0} and $\varepsilon\le \sigma$, we obtain $\ell(K^{y,\sigma})\ge 1-\varepsilon\ge 1-\sigma$.
    
    Secondly, the total variation estimate \eqref{ApproxCtrl-2} is a consequence of \cite[Theorem 2.4]{KNS-20}, once we can justify the assumptions (a) and (b) therein, with parameter $\gamma=1$ in (b). To verify (a), we point out that $\Phi^{y,x}$ vanishes on $K\setminus K^{y,\sigma}$, and the Lipschitz estimate on $K^{y,\sigma}$ follows from \eqref{ACL-1}, \eqref{ACL-3}, \eqref{Rdef}, \eqref{inverse-1} and the local Lipschitz-continuity of $A$. To verify (b), note that \cite[Corollary~7.7]{NZZ-24} is applicable, since $0$ is a regular value of \eqref{regularmap}, and $\mathfrak{F}_\varepsilon(y,\cdot)$ is locally Lipschitz in the sense of \eqref{mathfrakFLip}. Then the same arguments as in \cite[Section~2.6]{KNS-20}, with \cite[Corollary~7.7]{NZZ-24} in place of \cite[Corollary 3.3]{KNS-20}, would establish the assumption (b) with parameter $\gamma=1$.
    
    Finally, for local stabilization \eqref{ApproxCtrl-3}, assume $\|y-x\|\le \delta$ and $\zeta\in K^{y,\sigma}$, with $\delta\in (0,1)$ to be determined. Then we simply follow the reasoning in the proof of Lemma~\ref{lem_contoltoy} line-by-line, with $R^\gamma$ replaced by $R_\varepsilon(y,\zeta)$, to find that (cf.~\eqref{stabilization-1})
    \[\|S(y,\zeta)-S(x,\zeta+\Phi^{y,x}(\zeta))\|\le (q_0+C_0 \varepsilon +C_\varepsilon \delta) \|y-x\|.\]
    The main difference is the last constant $C_\varepsilon$ from \eqref{inverse-1}. Nevertheless, thanks to \eqref{epsilonsmall}, we can choose $\delta$ sufficiently small (depending only on $\varepsilon$) to ensure \eqref{ApproxCtrl-3}.
\end{proof}

\subsubsection{From control to coupling}\label{appendix_control->coupling}

We mimic \cite[Section 2.2]{KNS-20} to drive Lemma~\ref{lem_coupling} from Lemma~\ref{lem_approxcontrol}.

\begin{proof}[\rm {\bf Proof of ``Lemma~\ref{lem_approxcontrol} $\Rightarrow$ Lemma~\ref{lem_coupling}"}]
    Choose $\sigma=\nu/2\in (0,1)$ and $q_1=q_2\in (q_0,1)$ in Lemma~\ref{lem_approxcontrol}, which gives rise to notions $\delta$, $C_\sigma$, $\Phi$ and $K^{y,\sigma}$. Next we fix $d_\nu\in (0,1)$ so small that
    \begin{equation}\label{dsmall}
        d_\nu\le \delta\quad \text{and}\quad C_\sigma d_\nu^{1/2}\le \nu/2
    \end{equation}
    Let us show that for any $d\in (0,d_\nu]$, the desired coupling exists, and \eqref{coupling-2} holds with $C_\nu=C_\sigma$.
    
    If $(y,y')\in \mathbf{Y}_\infty\cup \mathbf{Y}_{<0}$, we obey the requirement in assertion (a) and set
    \[V(y,y')=S(y,\tilde{\eta})\quad \text{and}\quad V'(y,y')=S(y',\tilde{\eta}).\]
    Here $\tilde{\eta}$ is an i.i.d.~copy of random noise $\eta_n$, whose law is equal to $\ell$.
    
    Next we turn to the case $(y,y')\in \mathbf{Y}_n\, (0\le n<\infty)$. We can find $(\hat{\eta},\eta')\in \mathcal{C}((\Id+\Phi^{y,y'})_*\ell, \ell)$ as a maximal coupling (the dependence on $y,y'$ is hidden for simplicity), i.e.
    \begin{equation}\label{maximalcoupling}
        \mathbb{P}(\eta'\not =\hat{\eta})=\|\ell-(\Id+\Phi^{y,y'})_*\ell\|_{TV}.
    \end{equation}
    Moreover, $(\hat{\eta},\eta')$ can be taken to be measurable in $y,y'\in Y$; see, e.g., \cite[Theorem~1.2.28]{KS-12}. Since the law of $\tilde{\eta}+\Phi^{y,y'}(\tilde{\eta})$ is also equal to $(\Id +\Phi^{y,y'})_*\ell$, by virtue of a gluing lemma \cite[Theorem~7.1]{KNS-20}, we can construct a tuple of random variables $(\eta,\hat{\eta},\eta')$, so that $(\hat{\eta},\eta')$ is a maximal coupling between $(\Id+\Phi^{y,y'})_*(\ell)$ and $\ell$ as above, and
    \[\mathscr{D}(\eta,\hat{\eta})=\mathscr{D}(\tilde{\eta},\tilde{\eta}+\Phi^{y,y'}(\tilde{\eta})).\]
    In addition, $(\eta,\hat{\eta},\eta')$ remains measurable in $y,y'$. Note that $\mathscr{D}(\eta)=\mathscr{D}(\eta')=\ell$, and $\hat{\eta}=\eta+\Phi^{y,y'}(\eta)$ almost surely. We can thus define the coupling between $P(y,\cdot)$ and $P(y',\cdot)$ by
    \[V(y,y')=S(y,\eta)\quad \text{and}\quad V'(y,y')=S(y',\eta').\]
    Then the assertion (a) of Lemma~\ref{lem_coupling} follows immediately. It remains to verify (b). In the rest of the proof, we fix $y,y'\in Y$ with $\|y-y'\|\le d$, and omit the arguments $y,y'$ in $V,V'$.

    To this end, consider the partition of $\Omega$ by three events
    \[\Omega^{y,y'}_{\rm coup}=\{\eta \in K^{y,\sigma}\}\cap \{\eta'=\hat{\eta}\},\quad \Omega^{y,y'}_{\rm except}=\{\eta \not \in K^{y,\sigma}\}\cap \{\eta'=\hat{\eta}\},\quad \Omega^{y,y'}_{\rm uncoup}=\{\eta' \not=\hat{\eta}\}.\]
    Thanks to \eqref{ApproxCtrl-1}, \eqref{ApproxCtrl-2} and \eqref{dsmall}, the probability of these events can be estimated by
    \begin{equation}\label{probcoup}
        \begin{aligned}
            \mathbb{P}(\Omega^{y,y'}_{\rm coup})&\ge \ell(K^{y,\sigma})-\mathbb{P}(\eta'\not =\eta'')=\ell(K^{y,\sigma})-\|\ell-(\Id+\Phi^{y,y'})_*\ell\|_{TV}\\
            &\ge 1-\sigma-C_\sigma d^{1/2}\ge 1-\nu,
        \end{aligned}
    \end{equation}
    and
    \begin{equation}\label{probuncoup}
        \mathbb{P}(\Omega^{y,y'}_{\rm uncoup})=\mathbb{P}(\eta'\not =\eta'')\le C_\sigma \|x-x'\|^{1/2}.
    \end{equation}
    
    On the event $\Omega_{\rm coup}^{y,y'}$, by \eqref{ApproxCtrl-3} we have
    \[\|V-V'\|=\|S(y,\eta)-S(y',\eta+\Phi^{y,y'}(\eta))\|\le q_2 \|y-y'\|.\]
    This together with \eqref{probcoup} implies \eqref{coupling-1}. 
    
    Furthermore, on the event $\Omega^{y,y'}_{\rm except}$ we have $\Phi^{y,y'}(\eta)=0$ and thus $\eta'-\hat{\eta}=\eta$. Hence the Lipschitz-continuity \eqref{LipS} implies that, if $(y,y')\in \mathbf{Y}_n$ with $n\ge L$, then
    \[\{(V,V')\in \mathbf{Y}_{<n-L}\}\subset \Omega^{y,y'}_{\rm uncoup},\]
    Thus \eqref{coupling-2} follows from \eqref{probuncoup}.
    % \[\mathbb{P}((V,V')\in
    % \mathbf{Y}_{<n-L})\le \mathbb{P}(\Omega^{y,y'}_{\rm uncoup})\le C_\sigma \|x-x'\|^{1/2}.\]
    % In particular, this constant $C_\sigma$ is independent of $d$, and only depends on $\sigma=\nu/2$.
\end{proof}

\subsubsection{From coupling to mixing}\label{appendix_coupling->mixing}

We now tackle Proposition~\ref{prop_restricY}, relying on Lemma~\ref{lem_coupling} and some arguments from \cite[Section 2.2]{KNS-20}. Meanwhile, as the dissipation in hypothesis $(\mathbf{H3})$ only occurs at time $m_0$, we need to consider the $m_0$-th iteration of the coupling $(V,V')$.

To start with, we introduce the well-known Kantorovich functionals. Recall $(y_k,\mathbb{P}_y)$ denotes the Markov process on compact subset $Y$ in Hilbert space $X$. Consider a bounded symmetric Borel function $F\colon Y\times Y\to \mathbb{R}$, satisfying
\begin{equation}\label{KF-1}
    F(y,y')\ge C \|y-y'\|^\beta\quad \text{for any }y,y'\in Y,
\end{equation}
with constants $C>0$ and $\beta\in (0,1]$ independent of $y,y'$. The Kantorovich functional associated with $F$ is denoted by $\mathcal{K}_F\colon \mathcal{P}(Y)\times \mathcal{P}(Y)\to [0,\infty)$, and defined as (recall $\mathcal{C}$ refers to couplings)
\[\mathcal{K}_F(\mu_1,\mu_2):=\inf \{\mathbb{E} F(\xi_1,\xi_2): (\xi_1,\xi_2)\in \mathcal{C}(\mu_1,\mu_2)\}.\]

The following implication is well-known, and can be found in, e.g., \cite[Theorem 3.1.1]{KS-12}.
% By definition, for any $(\xi_1,\xi_2)\in \mathcal{C}(\mu_1,\mu_2)$,
% \[\|\mu_1-\mu_2\|_L^*=\sup_{\|f\|_L=1} |\mathbb{E}f(\xi_1)-\mathbb{E}f(\xi_2)|\le \mathbb{E} \|\xi_1-\xi_2\|\le C\mathbb{E} F(\xi_1,\xi_2)^{1/\beta}.\]
% Since $F$ is bounded and $1/\beta\ge 1$, we thus obtain
% \[\|\mu_1-\mu_2\|_L^*\le C \|F\|_{L^\infty}^{1/\beta-1} \mathcal{K}_F(\mu_1,\mu_2).\]
% This suggests to interpret exponential mixing as the exponential decay of Kantorovich functional. Indeed, we have the following result, which is a special case of \cite[Theorem 3.1.1]{KS-12}.

\begin{lemma}\label{lem_Kanfunctional}
    Under the above settings, suppose there exists $m\in \mathbb{N}$ and $c\in (0,1)$ such that
    \begin{equation}\label{KF-2}
        \mathcal{K}_F(P_m^* \mu_1,P_m^*\mu_2)\le c\mathcal{K}_F(\mu_1,\mu_2) \quad \text{for any }\mu_1,\mu_2\in \mathcal{P}(Y).
    \end{equation}
    Then the Markov process $(y_k,\mathbb{P}_y)$ admits a unique invariant measure $\mu\in \mathcal{P}(Y)$, and the exponential mixing property holds in the sense of \eqref{EM-2}.
\end{lemma}

We next invoke this lemma to derive Proposition~\ref{prop_restricY} from Lemma~\ref{lem_coupling}.

\begin{proof}[\rm \textbf{Proof of ``Lemma~\ref{lem_coupling} $\Rightarrow$ Proposition~\ref{prop_restricY}"}]
    Fix any $q_2\in (q,1)$ in Lemma~\ref{lem_coupling}, and choose $\nu\in (0,1)$ sufficiently small, so that
    \begin{equation}\label{c1def}
        c_1:=q_2^{m_0/2}+2m_0 \nu q_2^{-m_0 L/2}<1,
    \end{equation}
    where $m_0$ appears in hypothesis $(\mathbf{H3})$. Then fix $d\in (0,d_\nu]$ sufficiently small so that
    \begin{equation}\label{c2def}
        c_2:=q_2^{m_0/2}+2C_\nu m_0 q^{-m_0 L/2}\cdot d^{1/2}+m_0 \nu q_2^{-m_0 L/2}<1.
    \end{equation}
    Our goal is to construct a suitable Kantorovich functional, such that \eqref{KF-2} holds with $m=m_0$.
    
    First we introduce the iteration of the coupling map $V,V'$. Set the sample space $\boldsymbol{\Omega}=\Omega^{\mathbb{N}}$ equipped with product measure. For $\boldsymbol{\omega}=(\omega_n)_{n\in \mathbb{N}}\in \boldsymbol{\Omega}$, let $(V_0(y,y'),V_0'(y,y'))=(y,y')$, and iteratively define for $n\ge 1$ (the arguments $y,y',\boldsymbol{\omega}$ in $V_{n-1},V_{n-1}'$ are omitted for clarity)
    \[V_n(y,y',\boldsymbol{\omega})=V(V_{n-1},V_{n-1}',\omega_n)\quad \text{and}\quad V_n'(y,y',\boldsymbol{\omega})=V'(V_{n-1},V_{n-1}',\omega_n).\]
    Since $(V(y,y'),V'(y,y'))\in \mathcal{C}(P(y,\cdot),P(y',\cdot))$ for any $y,y'\in Y$, and due to the Kolmogorov--Chapman relation, it is easy to check by induction that 
    \[(V_n(y,y'),V_n'(y,y'))\in \mathcal{C}(P_n(y,\cdot),P_n(y',\cdot))\quad \text{for any }n\in \mathbb{N}.\]
    
    Next we construct a bounded Borel function $F\colon Y\times Y\to [0,\infty)$ by
    \[F(y,y')=\begin{cases}
        0,&(y,y')\in \mathbf{Y}_\infty,\\
        (q_2^n d)^{1/2},&(y,y')\in \mathbf{Y}_n,\ n\ge 0,\\
        (2-(2/p)^k)d^{1/2},&(y,y')\in \mathbf{Y}_k,\ -N+1\le k\le -1,\\
        2d^{1/2},&(y,y')\in \mathbf{Y}_{-N}.
    \end{cases}\]
    Here the parameter $p\in (0,1)$ will be chosen later; we also mention that $2-(2/p)^k>0$ since $k<0$. Note that $F$ is a constant on each $\mathbf{Y}_l\, (-N\le l\le \infty)$, which is denoted with $F_l$, and is decreasing with respect to $l$. One immediate checks \eqref{KF-1} with $\beta=1/2$, as
    \[F(y,y')^2\ge \begin{cases}
        \|y-y'\|,&(y,y')\in \mathbf{Y}_n,\ 0\le n\le \infty,\\
        (d/2R)\|y-y'\|,&(y,y')\in \mathbf{Y}_k,\ -N\le k\le -1.
    \end{cases}\]

    In order to establish \eqref{KF-2} with $m=m_0$, invoking a standard argument (see, e.g., \cite[Section 2.2]{KNS-20}), it suffices to consider the spacial case $\mu_1=\delta_y$ and $\mu_2=\delta_{y'}$. Namely, we need to show 
    \begin{equation}\label{KF-3}
        \mathbb{E} F(V_{m_0}(y,y'),V_{m_0}'(y,y'))\le c F(y,y')\quad \text{for any }y,y'\in Y.
    \end{equation}
    To this end, we examine the following four cases. As $y,y'\in Y$ are fixed, in the rest of the proof, we simply write $V_n,V_n'$ instead of $V_n(y,y'),V_n'(y,y')$.

    \medskip

    \textit{Case 0:} If $(y,y')\in \mathbf{Y}_\infty$, then Lemma~\ref{lem_coupling}(a) yields $V_n=V_n'$, and thus \eqref{KF-3} holds.

    \medskip

    \textit{Case 1:} If $(y,y')\in \mathbf{Y}_n$ for some $n\ge m_0 L$, divide $\boldsymbol{\Omega}$ into the following three events:
    \begin{itemize}
        \item[\tiny$\bullet$] $A_1:=\{(V_{m_0},V'_{m_0})\in \mathbf{Y}_{\ge n+m_0}\}$, on which $F(V_{m_0},V'_{m_0})\le q_2^{m_0/2} F(y,y')$. By \eqref{coupling-1},
        \begin{align*}
            \mathbb{P}(A_1)&\ge \mathbb{P}((V_j,V'_j)\in \mathbf{Y}_{n+j}\ \text{for any }1\le j\le m_0)\\
            &=\prod_{j=1}^{m_0} \mathbb{P}((V_j,V'_j)\in \mathbf{Y}_{\ge n+j}|(V_k,V'_k)\in \mathbf{Y}_{\ge n+k}\ \text{for any }1\le k\le j)\\
            &\ge \prod_{j=1}^{m_0} \mathbb{P}((V_j,V'_j)\in \mathbf{Y}_{\ge n+j}|(V_{j-1},V'_{j-1})\in \mathbf{Y}_{\ge n+j-1})\ge (1-\nu)^{m_0}\ge 1-m_0 \nu.
        \end{align*}
        % Here for the first inequality in the last line, we use an easy consequence of \eqref{coupling-1}:
        % \begin{align*}
        %     \mathbb{P}((V_{j},V_{j}')\in \mathbf{Y}_{\ge n+j})&\ge \sum_{k=n+j}^\infty \mathbb{P}((V_j,V_j')\in \mathbf{Y}_k,\ (V_{j-1},V_{j-1}')\in \mathbf{Y}_{k-1})\\
        %     &\ge \sum_{k=n+j}^\infty (1-\nu)\mathbb{P}((V_{j-1},V_{j-1}')\in \mathbf{Y}_{k-1})\\
        %     &=(1-\nu)\mathbb{P}((V_{j-1},V_{j-1}')\in \mathbf{Y}_{\ge n+j-1}).
        % % \end{align*}
        % Also note that, here and in the rest of the proof, the summations actually contains the term when $k=\infty$. It it easy to see that this bring no difference to the reasoning.

        \item[\tiny$\bullet$] $B_1:=\{(V_{m_0},V_{m_0}')\in \mathbf{Y}_{<n-m_0 L}\}$, on which $F(y,y')\le F_{-N}=2d^{1/2}$. By \eqref{coupling-2},
        \begin{align*}
            \mathbb{P}(B_1)&\le \mathbb{P}(\exists 1\le j\le m_0,\ (V_j,V_j')\in \mathbf{Y}_{<n-jL})\\
            &\le \sum_{j=1}^{m_0} \mathbb{P}((V_{j-1},V_{j-1}')\in \mathbf{Y}_{\ge n-(j-1)L},\ (V_j,V_j')\in \mathbf{Y}_{<n-jL})\\
            &\le \sum_{j=1}^{m_0} C_\nu (q_2^{-(j-1)L+1}\|y-y'\|)^{1/2}\le C_\nu m_0 q_2^{-m_0 L/2}\|y-y'\|^{1/2}.
        \end{align*}
        % Here in the last line we use a similar argument as in the estimate for $\mathbb{P}(A_1)$:
        % \begin{align*}
        %     &\mathbb{P}((V_{j-1},V_{j-1}')\in \mathbf{Y}_{\ge n-(j-1)L},\ (V_j,V_j')\in \mathbf{Y}_{<n-jL})\\
        %     \le &\sum_{k=n-(j-1)L}^\infty \mathbb{P}((V_{j-1},V_{j-1}')\in \mathbf{Y}_k,\ (V_j,V_j')\in \mathbf{Y}_{k-L})\\
        %     \le& \sum_{k=n-(j-1)L}^\infty  C_\nu (q_2^k d)^{1/2}\mathbb{P}((V_{j-1},V_{j-1}')\in \mathbf{Y}_k))\\
        %     \le &C_\nu (q_2^{n-(j-1)L} d)^{1/2}\le C_\nu (q_2^{-(j-1)L+1}\|y-y'\|)^{1/2}.
        % \end{align*}

        \item[\tiny$\bullet$] $D_1:=\Omega\setminus (A_1\cup B_1)$, on which $F(V_{m_0},V_{m_0}')\le q_2^{-m_0L/2} F(y,y')$, and 
        \[\mathbb{P}(D_1)\le 1-\mathbb{P}(A_1)\le m_0 \nu.\]
    \end{itemize}
    Thanks to $F(y,y')\ge \|y-y'\|^{1/2}$ and \eqref{c1def}, we thus obtain
    \begin{align*}
        \mathbb{E} F(V_{m_0},V_{m_0}')&\le \mathbb{P}(A_1)\sup_{A_1} F(V_{m_0},V_{m_0}')+\mathbb{P}(B_1)\sup_{B_1} F(V_{m_0},V_{m_0}')+\mathbb{P}(D_1)\sup_{D_1} F(V_{m_0},V_{m_0}')\\
        &\le 1\cdot q_2^{m_0/2} F(y,y')+C_\nu m_0 q_2^{-m_0 L/2}\|y-y'\|^{1/2}\cdot 2d^{1/2}+m_0\nu \cdot q_2^{-m_0L/2} F(y,y')\\
        &\le (q_2^{m_0/2}+2C_\nu m_0 q^{-m_0 L/2}\cdot d^{1/2}+m_0 \nu q_2^{-m_0 L/2})F(y,y')=c_2 F(y,y').
    \end{align*}

    \medskip

    \textit{Case 2:} If $(y,y')\in \mathbf{Y}_n$ for some $0\le n\le m_0 L-1$, divide $\boldsymbol{\Omega}$ into the following two events:
    \begin{itemize}
        \item[\tiny$\bullet$] $A_2:=(V_{m_0},V_{m_0}')\in \mathbf{Y}_{\ge n+m_0}$, on which $F(V_{m_0},V_{m_0}')\le q_2^{m_0/2} F(y,y')$. The same argument for the estimate of $\mathbb{P}(A_1)$ in Case 1 also yields $\mathbb{P}(A_2)\ge 1-m_0 \nu$.

        \item[\tiny$\bullet$] $D_2:=(V_{m_0},V_{m_0}')\in \mathbf{Y}_{<n+m_0}$. Then $\mathbb{P}(D_2)\le 1-\mathbb{P}(A_2)\le m_0\nu$, and on $D_2$ we have
        \[F(V_{m_0},V_{m_0}')\le 2d^{1/2}\le 2(q_2^{-(n+1)}\|y-y'\|)^{1/2}\le 2q_2^{-m_0L/2} F(y,y').\]
    \end{itemize}
    Hence in this case, thanks to \eqref{c2def}, we obtain
    \begin{align*}
        \mathbb{E} F(V_{m_0},V_{m_0}')&\le \mathbb{P}(A_2) \sup_{A_2} F(V_{m_0},V_{m_0}')+\mathbb{P}(D_2) \sup_{D_2} F(V_{m_0},V_{m_0}')\\
        &\le 1\cdot q_2^{m_0/2}F(y,y')+m_0\nu\cdot 2 q_2^{-m_0L/2} F(y,y')=c_1 F(y,y').
    \end{align*}

    \medskip

    \textit{Case 3:} If $(y,y')\in \mathbf{Y}_k$ for some $-N\le k\le -1$, we claim that there exists a constant $p>0$ (independent of $y,y'$ and $k$), such that the event $A_3:=\{(V_{m_0},V_{m_0}')\in \mathbf{Y}_{\ge k+1}\}$ satisfies $\mathbb{P}(A_3)\ge p$. We use this constant $p$ in the definition of $F$. Once the claim is true, then the monotonicity of $F_l$ implies
    \begin{align*}
        \mathbb{E}F(V_{m_0},V_{m_0}')&\le F_{l+1}\mathbb{P}(A_3)+F_{-N}\mathbb{P}(A_3^C)\\
        &\le p\cdot(2-(2/p)^{k+1})d^{1/2}+(1-p)\cdot 2d^{1/2}\\
        &=(2-2\cdot(2/p)^k)d^{1/2}\le c_3 F(y,y'),
    \end{align*}
    Here we tacitly use the fact that $F_k=(2-(2/p)^k) d^{1/2}$ is also valid for $k=0$, and
    \[c_3:=\max_{-N\le k\le -1} \frac{2-2\cdot (2/p)^k}{2-(2/p)^k}\in (0,1).\]
    
    In conclusion, \eqref{KF-3} holds with $c:=\max\{c_1,c_2,c_3\}\in (0,1)$, and thus Lemma~\ref{lem_Kanfunctional} immediately leads to Proposition~\ref{prop_restricY}. It remains to demonstrate our claim in Case 3.
\end{proof}

\begin{proof}[\rm \textbf{Proof of $\mathbb{P}(A_3)\ge p>0$ in Case 3}]
    Consider the event
    \[B:=\{(V_j,V_j')\in \mathbf{Y}_{<0}\text{ for any }1\le j<m_0\}.\]
    Then due to Lemma~\ref{lem_coupling}(a), there are i.i.d.~copies of random variables $\tilde{\eta}_1,\dots ,\tilde{\eta}_{m_0}$ with common law $\ell$, so that on the event $B$, almost surely
    \[V_{m_0}=S_{m_0}(y;\tilde{\eta}_1,\dots ,\tilde{\eta}_{m_0})\quad \text{and}\quad V_{m_0}'=S_{m_0}(y';\tilde{\eta}_1,\dots ,\tilde{\eta}_{m_0}).\]
    With a small parameter $\varepsilon>0$ to be determined, we introduce another auxiliary event
    \[D:=\{\|\tilde{\eta}_j-\tilde{\zeta}_j\|<\varepsilon\text{ for any }1\le j\le m_0\}.\]
    Here $\tilde{\zeta}_j\in K$ are the same as in $(\mathbf{H3})$. Note that $\mathbb{P}(D)>0$ since $\tilde{\zeta}_1,\dots ,\tilde{\zeta}_{m_0}\in \supp(\ell)$.
    
    On the one hand, on event $B^C$ there exists $1\le j<m_0$ so that $(V_j,V_j')\in \mathbf{Y}_{\ge 0}$. Similar to our analysis on $\mathbb{P}(A_1)$ in Step 1, it is easy to find that
    \begin{equation}\label{P(A)-1}
        \mathbb{P}(A_3)\ge (1-\nu)^{m_0} \mathbb{P}(B^C)
    \end{equation}
    
    On the other hand, since $S$ is locally Lipschitz, on the event $B\cap D$ we have
    \begin{align*}
        \|V_{m_0}-\tilde{y}\|&\le \|S_{m_0}(y;\tilde{\zeta}_1,\dots ,\tilde{\zeta}_m)-\tilde{y}\|+C\sum_{j=1}^{m_0} \|\zeta_j-\tilde{\zeta}_j\|_E\\
        &\le \left(q_0+Cd^{-1}m_0 \varepsilon \right)(\|y-\tilde{y}\|\vee \|y'-\tilde{y}\|).
    \end{align*}
    Here we tacitly use that $\|y-y'\|>d$ implies $\|y-\tilde{y}\|\vee \|y'-\tilde{y}\|\ge d/2$. And the same estimate holds for $\|V_{m_0}'-\tilde{y}\|$. If we choose $\varepsilon\ll 1$ so that $q_0+C d^{-1} m_0 \varepsilon\le q_2$, then
    \begin{equation}\label{P(A)-2}
        \mathbb{P}(A)\ge \mathbb{P}(B\cap D)\ge \mathbb{P}(D)-\mathbb{P}(B^C).
    \end{equation}
    
    Combining the estimates from two aspects \eqref{P(A)-1} and \eqref{P(A)-2}, the claim is justified as
    \[\mathbb{P}(A)\ge \frac{1\cdot (1-\nu)^{m_0} \mathbb{P}(B^C)+(1-\nu)^{m_0}\cdot (\mathbb{P}(D)-\mathbb{P}(B^C))}{1+(1-\nu)^{m_0}}=\frac{(1-\nu)^{m_0}\mathbb{P}(D)}{1+(1-\nu)^{m_0}}=:p>0.\]
    This is a constant determined by $\nu,m_0,d,\varepsilon$ and the law of random variables $\ell$. These parameters are all fixed beforehand. Specifically, $p$ is independent of $y,y'\in \mathbf{Y}_{<0}$.
\end{proof}

\subsection{Elements for the analysis of Schr\"odinger equations}

\subsubsection{Bourgain spaces and basic estimates}\label{appendix_Bourgain}

The Bourgain spaces, introduced by Bourgain in \cite{Bourgain-93}, is defined in Definition~\ref{def_Bourgain}. We recall here some basic properties and multilinear estimates needed in this paper. Most of the results are well-known, and can be found in e.g., \cite{Bourgain-book,RZ-09,Laurent-ECOCV,CXZZ-25}.

\begin{lemma}\label{lemBourgain_embed}
    For any $T>0$ and $s\in \mathbb{R}$, if $b>1/2$, then $X^{s,b}_T\hookrightarrow C(0,T;H^s(\mathbb{T}))$.
\end{lemma}

\begin{lemma}\label{lemBourgain_S(t)estimate}
For any $T>0$, $s\in \mathbb{R}$ and $b\in (1/2,1)$, there exists a constant $C>0$ so that
\begin{equation}\label{Multilinear-estimate-2}
\|S_a(t)u_0\|_{X_T^{s,b}}\leq C\|u_0\|_{H^s}\quad \text{for any }u_0\in H^s(\mathbb{T}),
\end{equation}
\begin{equation}\label{Multilinear-estimate-3}
\left\|\int_0^tS_a(t-\tau)F(\tau)d\tau\right\|_{X_T^{s,b}}\leq C\|F\|_{X_T^{s,b-1}}\quad \text{for any }f\in X_T^{s,b-1}.
\end{equation}
\end{lemma}
    
    The proof can be found in \cite[Lemma 4.1 and Lemma 4.2]{RZ-09}. We mention that without the damping $a(x)$, these estimates with respect to $S(t)$ are standard; see, e.g., \cite{Laurent-ECOCV}.

% In applications, the function $F$ in \eqref{Multilinear-estimate-3} represents either the external force or nonlinear terms. The next lemma of multilinear estimates serves to deal with the nonlinear term $|u|^{p-1} u$. Recall that $\mathcal{N}(u_1,\dots ,u_p)$ refers to the $p$-linear operator defined by \eqref{p_product}.

% \begin{lemma}\label{lemBourgain_multilinear}
% For every $T>0$, $s\ge 1$ and constants $b,b'$ satisfying $0<b'<1/2<b<1$ and $b+b'\le 1$, there exists a constant $C>0$, such that for any $u_1,\cdots,u_p\in X_T^{s,b}$,
% \begin{equation}\label{Multilinear-estimate}
% \|\mathcal{N} (u_1,\cdots,u_p)\|_{X_T^{s,-b'}}\leq C\sum_{j=1}^p \|u_j\|_{X_T^{s,b}} \prod_{l\not =j}\|u_l\|_{X_T^{1,b}}.
% \end{equation}
% \end{lemma}

% When $p=3$, these estimate are well-known, and valid for $s\ge 0$; see, e.g., \cite[Section V.2]{Bourgain-book}. As for general odd $p\ge 3$, we refer the reader to \cite[Proposition A.4]{CXZZ-25}.

\begin{lemma}\label{Lemma-multiplication}
For every $T>0$, $s\in \mathbb{R}$, $b\in [-1,1]$, there exists a constant $C>0$ such that
\[\|\psi(t)u\|_{_{X_T^{s,b}}}\le C\|\psi\|_{_{H^1(0,T;\mathbb{C})}}\|u\|_{_{X_T^{s,b}}}\quad \text{for any }\psi\in H^1(0,T;\mathbb{C}),\ u\in X_T^{s,b}.\]
\end{lemma}

The proof is verbatim as \cite[Lemma 1.2]{Laurent-ECOCV}, despite $\psi$ is not assumed to be smooth.

\subsubsection{Well-posedness of linearized and backward systems}

We state the well-posedness concerning two linear systems derived from the NLS equations. The first is the linearized equation \eqref{D_uS}, and the second is the adjoint backward system \eqref{adjointsystem}.

\begin{lemma}[{\cite[Proposition A.7]{CXZZ-25}}]\label{lem_wplinearization}
    Let $s\ge 1$ and $R>0$ be arbitrarily given. Assume the constants $b,b'$ satisfy $0<b'<1/2<b$ and $b+b'<1$. Then there exists a constant $C>0$, such that for any $u\in X_1^{s,b}$ with $\|u\|_{X_1^{s,b}}\le R$, the following assertions hold.

    \begin{enumerate}
        \item[(a)] For any $v_0\in H^s$, the linearized equation \eqref{D_uS} admits a unique solution $v\in X_1^{s,b}$, and
        \begin{equation*}
            \|v\|_{X^{s,b}_1}\leq C\|v_0\|_{H^s}.
        \end{equation*}

        \item[(b)] For any $\varphi_1\in H^{-s}$, the backward system \eqref{adjointsystem} admits a unique solution $\varphi\in X_1^{-s,b}$, and
        \begin{equation*}
            \|\varphi\|_{X^{-s,b}_1}\leq C\|\varphi_1\|_{H^{-s}}.
        \end{equation*}
    \end{enumerate}
\end{lemma}

\end{appendices}

\medskip

\noindent\textbf{Funding} \; Shengquan Xiang is partially supported by NSFC 12571474. Zhifei Zhang is partially supported by NSFC 12288101.

\noindent\textbf{Acknowledgments} \;
The authors would like to thank Ziyu Liu for valuable discussions and suggestions during the preparation of the paper.

\normalem
\bibliographystyle{plain}
\bibliography{References}
	
\end{document}